\documentclass[final]{amsart}
\usepackage{amsthm}
\usepackage{amsmath}
\usepackage{amsfonts}
\usepackage{amssymb}
\usepackage[usenames]{color}
\usepackage[mathscr]{euscript}
\usepackage[hidelinks]{hyperref}
\usepackage[arrow,curve,frame,color]{xy}
\def\hmeas#1.{\mathscr{H}^{\setbox0=\hbox{$#1\unskip$}\ifdim\wd0=0pt 1
    \else #1\fi}} 
\def\real{{\mathbb{R}}}
\def\zahlen{{\mathbb Z}}
\def\glip#1.{{\bf L}(#1)} 
\DeclareMathOperator\diam{diam}
\DeclareMathOperator\im{Im}
\DeclareMathOperator\Inside{Int}
\def\natural{{\mathbb N}}
\DeclareMathOperator\dom{dom} 
\newcommand{\on}{\:\mbox{\rule{0.1ex}{1.2ex}\rule{1.1ex}{0.1ex}}\:}
\def\cmass#1.{{\|#1\|}}
\newcommand{\dist}{\operatorname{dist}}

\numberwithin{equation}{section} 
\theoremstyle{plain}
\newtheorem{lem}[equation]{Lemma}

\newtheorem{thm}[equation]{Theorem}

\theoremstyle{definition}
\newtheorem{defn}[equation]{Definition}
\newtheorem{constr}[equation]{Construction}
\theoremstyle{remark}

\begin{document}
\title{Unrectifiable normal currents in Euclidean spaces}
\author{Andrea Schioppa}
\email{andrea.schioppa@math.ethz.ch}
\begin{abstract}
  We construct in $\real^{k+2}$ a $k$-dimensional simple normal
  current whose support is purely $2$-unrectifiable. The result is
  sharp because the support of a normal current cannot be purely
  $1$-unrectifiable and a $(k+1)$-dimensional normal current can be
  represented as an integral of $(k+1)$-rectifiable currents. This
  gives a negative answer to the (revised version) of a question of Frank Morgan (1984).
\end{abstract}
\maketitle
\tableofcontents
\section{Introduction}
\label{sec:intro}
\subsection{Results}
\label{subsec:results}
This paper is a continuation of \cite{schioppa_metric_sponge} to which we
refer for more background and notation. The main motivation behind
\cite{schioppa_metric_sponge} was to provide new examples of
Ambrosio-Kirchheim metric currents~\cite{ambrosio-kirch} and to prove that
higher-dimensional analogues of some results in
\cite{deralb,curr_alb} \emph{do not} hold. Specifically, in
\cite{deralb} it was shown that in metric measure spaces
vector fields can be concretely described as a superposition of
partial derivative operators associated with curve fragments. In
particular, the background measure $\mu$ appearing in the definition
of vector fields (see for example Subsec.~2.1 in
\cite{schioppa_metric_sponge} about Weaver derivations) has to
admit Alberti representations or, more precisely, has to be
\textbf{$1$-rectifiably representable}; this means that $\mu$ can be
represented as an integral of $1$-dimensional Hausdorff measures associated with curve
fragments $\gamma$: $\mu=\int \hmeas 1.\on \gamma dQ(\gamma)$. In the
case in which a higher order representation exists, i.e.~$\mu=\int
\hmeas k.\on\sigma dQ(\sigma)$ where $\sigma$ is a $k$-rectifiable
compact set (see \cite{ambrosio-rectifiability} for the theory of
rectifiable sets in metric spaces) we will say that $\mu$ is
\textbf{$k$-rectifiably representable}.
\par In \cite{curr_alb} it was later shown that $1$-dimensional
metric currents admitted an integral representation in terms of
$1$-rectifiable metric currents $T=\int [[\gamma]]\,dQ(\gamma)$
($[[\gamma]]$ being the current associated to an oriented fragment) and
that $k$-dimensional metric currents could be canonically associated to
$k$-dimensional vector fields obtaining a parallel between the metric
theory of Ambrosio-Kirchheim~\cite{ambrosio-kirch} and the classical
theory of Federer
and Fleming~\cite[Ch.~4]{federer_gmt}. A natural question we had at the time was whether a
$2$-dimensional metric
current $T$ could be represented as an integral of $2$-rectifiable
currents $T=\int [[\sigma]]\,dQ(\sigma)$. Some specific examples of
\emph{non-simple} (i.e.~the associated vector fields are not simple)
$2$-dimensional currents with $2$-purely unrectifiable supports had
been obtained by Marshall Williams~\cite{williams-currents} in Carnot groups. In
\cite{schioppa_metric_sponge} we obtained a general negative answer
constructing for each $k$ a simple $k$-dimensional normal metric
current whose support is purely $2$-unrectifiable. Unfortunately, those
currents could not be constructed in Euclidean spaces. In this paper
we complete the treatment by:
\begin{constr}
  \label{constr:main_result}
  In $\real^{k+2}$ there is a $k$-dimensional normal current whose
  support is purely $2$-unrectifiable.
\end{constr}
Note that our normal currents are also classical normal currents, thus
providing examples of normal currents which live on $2$-unrectifiable
subsets. 
\subsection{Relation to previous work}
\label{subsec:prev}
Even though I came across this problem while finishing my dissertation
in 2014, I later found that other researchers had previously
considered it. I learned from Giovanni Alberti that he had considered
also this problem, and later found out the following question of Frank
Morgan~\cite[Problem~3.8, pg.~446]{morgan-plist}:
\begin{description}
\item[(Q-Morgan)] Question of Frank Morgan. ``Can every normal current in
  $\real^s$ be decomposed as a convex integral combination of integral
  currents? In codimension one the answer is \emph{yes} if $\partial
  T=0$, see~\cite[\# 4.5.9(13)]{federer_gmt}''. (In my own words): Is a $k$-dimensional
  normal current $T$ in $\real^s$ representable as an integral
  $\int[[\sigma]]\,dQ(\sigma)$ of $k$-integral currents enforcing the
  \textbf{mass constraint} (here is what ``convex'' probably means)
 $\cmass T.=\int\cmass[[\sigma]].\,dQ(\sigma)$? Here $\cmass T.$
 denotes the mass measure of $T$.
\end{description}
The answer to \textbf{(Q-Morgan)} for $k=1$ is positive
by the beautiful work of Stanislav Smirnov~\cite{smirnov-solenoidal}, 
and recently we have learned from Alberti
and Massacesi that this is also the case for $k=s-1$ as a consequence
of the coarea formula for BV functions~\cite{massaccesi_phd,ambrosio-bvbook}: 
essentially they find a ``good filling'' for
the boundary of the normal current to reduce the problem to the case
$\partial T=0$ sketched by Morgan. 
\par In general the answer to \textbf{(Q-Morgan)} is negative:
Zworski~\cite{zworski-normal} gives as counterexample $T=\xi\,\hmeas
s.$ where $\xi$ is a suitable non-involutive $k$-field (a small gap in
his argument is pointed out and fixed
in~\cite[Chap.~2]{massaccesi_phd}). However, these examples are still
representable as integrals of integral currents if one drops the mass
constraint, and if one wants to keep the mass constraint, one can use
a remarkable Theorem of Alberti~\cite{alberti-lusin} to obtain an integral
decomposition into rectifiable currents by finding rectifiable sets
tangent to the non-involutive distribution. This suggests the following
revised version of \textbf{(Q-Morgan)}:
\begin{description}
\item[(Q-MorganRev)]  Is a $k$-dimensional
  normal current $T$ in $\real^s$ representable as an integral
  $\int[[\sigma]]\,dQ(\sigma)$ of \emph{$k$-rectifiable currents}
  \emph{without} necessarily satisfying the mass constraint?
\end{description}
Our result answers
\textbf{(Q-MorganRev)} in the negative for all $k\le s-2$: the support of a
$k$-current does not need even to intersect a $2$-rectifiable set in
positive area.
\par Our work has also applications to the recent structure theory for
measures developed in~\cite{acp_proceedings}. In particular, this answers the
problem of whether measures that admit a $k$-tangent field (this
essentially gives the directions along which a Rademacher Theorem on
the differentiability of Lipschitz functions holds) in the
sense of~\cite{acp_proceedings} are $k$-rectifiably
representable. Following~\cite{curr_alb,rindler-afree} we rephrase the problem
in the language of normal currents:
\begin{description}
\item[(Q-ACP)] Question of Alberti, Cs\"ornyei and
  Preiss~\cite[Sec.~2]{acp_proceedings}. 
If $\mu$ is a Radon measure on $\real^s$
  and for $1<k\le s$ there are $k$ $1$-dimensional normal currents
  $\{N_i\}_{i=1}^k$ with $\mu\ll \cmass N_i.$ and such that at
  $\mu$-a.e.~point the
  vector fields associated to the $N_i$ are independent, is then $\mu$
  $k$-rectifiably representable?
\end{description}
For $k=s$ \textbf{(Q-ACP)} has a positive answer by the recent work of
de Philippis and Rindler~\cite{rindler-afree}. For $s=3$ and $k=2$ a negative result has been
been announced by Andras Mathe~\cite{mathe-flow}. Our construction answers
\textbf{(Q-ACP)} in the negative for all $k\in\{2,\cdots,s-2\}$. It is
likely that modifications to our approach can also yield the negative answer
for $k=s-1$, but we do not pursue it further because it is likely to
follow also from the announced results of \cite{mathe-flow}.

\subsection{Organization}
\label{subsec:organization}
\par In the paper we follow the same approach in which we discovered
the result: there are the following $3$-layers:
\begin{description}
\item[Layer 1] A $2$-normal current in the Hilbert space $l^2$ whose support is purely
  $2$-unrectifiable.
\item[Layer 2] A $2$-normal current in $\real^4$ whose support is
  purely $2$-unrectifiable.
\item[Layer 3] A $k$-normal current in $\real^{k+2}$ whose support is
  purely $2$-unrectifiable.
\end{description}
\par \textbf{Layer 1} (Sec.~\ref{sec:2hilb}) is already non-trivial
because the Hilbert space has
the Radon-Nikodym property, i.e.~Lipschitz Hilbert-valued functions are differentiable
a.e. It is not hard to show that this implies that the examples in
\cite{schioppa_metric_sponge} cannot be bi-Lipschitz embedded in Hilbert
space. However, we are able to find a topological embedding of those
examples which is Lipschitz; an examination of the construction allows
to find a ``rate of collapse'' of the fibers of the double covers used
in \cite{schioppa_metric_sponge} which allows to prove
$2$-unrectifiability. Unfortunately, the Radon-Nikodym property prevents the use
of a simple blow-up argument as in \cite{schioppa_metric_sponge} and we must
resort to a quantitative estimate based on holonomy.
\par In \textbf{Layer 2} (Sec~\ref{sec:2r4}) we pass from Hilbert space to $\real^4$ by resorting to kernel
methods (see for example~\cite[Ch.~5]{mohri-found},
\cite{Guyon-automaticcapacity}) 
which are well-known in the SVMs
literature. Essentially the kernel trick allows to train an
SVM on an $\infty$-dimensional \emph{implicit} 
set of features even though the data set has (obviously)
only features living in a finite dimensional space. For example, in
$\real^4$ we can fabricate something like the Hilbert
space $l^2$ (countable sequences) using kernel
functions. Unfortunately, this
approach destroys the
approximate ``self-similarity'' of the construction in Hilbert space
making the details more technical and lengthy. In particular, we must resort to
curvilinear $(1+\varepsilon)$-Lipschitz projections to resolve the
fine structure of the support of the current at a given scale. 
\par In \textbf{Layer 3} (Sec~\ref{sec:krk}) we obtain the general case using a simple
idea from \cite{schioppa_metric_sponge} (I am indebted to Bruce
Kleiner for it) which consists in destroying
Lipschitz surfaces which are graphs on any pair of coordinate axes.
\subsection{Notational conventions}
\label{subsec:conventions}
For notational conventions, background and terminology we refer the reader to
\cite[Sec.~2]{schioppa_metric_sponge}. Here we use a more general notion of
weak* convergence for Lipschitz functions.
\begin{defn}[Weak* convergence for Lipschitz maps]
  \label{defn:weak*conv}
  Let $\{f_n\}_n$ be a sequence of Lipschitz maps $f_n:X\to Y$. We say
  that $f_n$ converges to a Lipschitz map $f:X\to Y$ in the \textbf{weak*
  sense} (and write $f_n\xrightarrow{\text{w*}} f$) if $f_n\to f$
  pointwise and $\sup_n\glip f_n.<\infty$, where $\glip f_n.$ denotes
  the Lipschitz constant of $f_n$.
  \par Assume that the sets $X_n\subset Z$ converge to the set
  $X\subset Z$ in the Hausdorff sense. For $x\in X$ we say that
  $\{x_n\}_n\subset Z$ with $x_n\in X_n$ \textbf{represents} $x\in X$ if
  $x_n\to x$.  Let $\{f_n\}_n$ be a sequence of Lipschitz maps
  $f_n:X_n\to Y$. We say
  that $f_n$ converges to a Lipschitz map $f:X\to Y$ in the \textbf{weak*
  sense} (and write $f_n\xrightarrow{\text{w*}} f$) if $\sup_n\glip
f_n.<\infty$, and whenever $\{x_n\}_n$ represents $x$, $f_n(x_n)\to
f(x)$. 
\par Note that in the previous definition one may check, for each $x$,
that $f_n(x_n)\to f(x)$ just for one sequence $\{x_n\}_n$ representing
$x$, thanks to the uniform bound on the Lipschitz constants of the
functions $f_n$.
\par Now assume also that the sets $Y_n\subset W$ converge to the set
  $Y\subset W$ in the Hausdorff sense.  We say
  that a squence $f_n:X_n\to Y_n$ converges to a Lipschitz map $f:X\to Y$ in the \textbf{weak*
  sense} (and write $f_n\xrightarrow{\text{w*}} f$) if $\sup_n\glip
f_n.<\infty$, and whenever $\{x_n\}_n$ represents $x$, $f_n(x_n)$
represents $f(x)$.
\end{defn}
In this paper there are only a couple of points where we use measured
Gromov-Hausdorff convergence. For background and notational
conventions we refer to \cite[Subsec.~3.1]{dim_blow}. However,
here we always reduce to the classical case by assuming that
convergence takes place in a \textbf{container} $Z$: if $(X_n,\mu_n)$
converges to $(X,\mu)$ in the measured Gromov-Hausdorff sense, we
assume that $X_n$ and $X$ are isometrically embedded in $Z$, and then
that $X_n\to X$ in the Hausdorff sense and $\mu_n\to\mu$ in the weak*
sense for Radon measures (i.e.~as functionals on continuous functions
defined on $Z$ which are bounded and have bounded support).
\par Finally, we use the convention $a\simeq b$ (or $a\approx b$) to say that $a/b,b/a\in[C^{-1},C]$
where $C$ is a universal constant; when we want to highlight $C$ we
write $a\simeq_Cb$. We similarly use notations like $a\lesssim b$ and
$a\gtrsim_Cb$.
\subsection*{Acknowledgements}
\label{subsec:ackwn}
This work has been partially supported by by the ``ETH Zurich Postdoctoral Fellowship Program and the Marie Curie Actions
  for People COFUND Program''.
\section{$2$-current in Hilbert space}
\label{sec:2hilb}
\def\sqfm#1,#2.{\setbox1=\hbox{$#1$\unskip}\setbox2=\hbox{$#2$\unskip}\text{\normalfont
    Sq}_{\ifdim\wd1>0pt #1\else i\fi}({\ifdim\wd2>0pt #2\else Q\fi}) }
\def\odec{_\text{\normalfont o}}
\def\cdec{_\text{\normalfont c}}
\def\adec{_\text{\normalfont a}}
Let $\{X_i\}_i$ denote the inverse system of square complexes in
\cite[Sec.~4]{schioppa_metric_sponge}, denote by
$X_\infty$ the corresponding inverse limit, and for $m\le n$
($n=\infty$ being allowed) let $\pi_{n,m}:X_n\to X_m$  denote the corresponding
$1$-Lipschitz projection. We let $\delta_n\searrow 0$ denote a
sequence with $\sum_n\delta_n=\infty$ and $\sum_n\delta_n^2<\infty$:
the precise form of $\delta_n$ will be determined later.
\par We briefly recall how $X_{i+1}$ is obtained out of $X_i$. 
Let $\sqfm ,X_i.$ denote the set of squares of
generation $i$ of $X_i$, whose side length is $l_i=5^{-i}$. To get
$X_{i+1}$ one subdivides each square $Q\in\sqfm ,X_i.$ and applies the
following operation. The square $Q$ is subdivided into squares of
generation $i+1$; there are $5^2$ such squares that, up to idenfying
$Q$ with $[0,5]^2$, can be indexed by the location of their south-west
corner by pairs $(j_1,j_2)\in\left\{0,\cdots, 4\right\}^2$.
These squares are grouped into three pieces: 
\begin{itemize}
\item The central square $Q\cdec$ corresponding to
  $(j_1,j_2)=(2,2)$.
\item The outer annulus $Q\odec$ corresponding to the squares where
  either $j_1\in\{0,4\}$ or $j_2\in\{0,4\}$.
\item The middle annulus $Q\adec$ consisting of the squares neither in
  $Q\cdec$ nor in $Q\odec$.
\end{itemize}
We make the simple observation $\hmeas
2.(Q\adec)\ge\frac{8}{25} \hmeas 2.(Q)$ and replace $Q\adec$ by a double cover
$\tilde Q\adec$, split the Lebesgue measure on $Q\adec$ in half and glue $\tilde Q\adec$
back to $Q\cdec$ and $Q\odec$ by collapsing the fibers of the cover on
the boundary $\partial \tilde Q\adec$ to match $\partial Q\cdec$ and the
inner component of $\partial Q\odec$. Let
$\tilde Q$ denote the square-complex thus obtained.
\begin{constr}[A map $\Psi:\tilde Q\to\real^2$ depending on a parameter
  $\delta$]
  \label{constr:basic_map}
  Fix $\delta>0$ small.
  Let $\hat Q\adec\subset Q\adec$ be the central annulus of the first
  subdivision of $Q\adec$ consisting of those squares in $\sqfm
  i+2,Q\adec.$ which are at distance $\ge 5^{-i-2}$ from $\partial
  Q\adec$. We observe that: $\hmeas 2.(\hat
  Q\adec)\ge\frac{3}{5}\hmeas 2.(Q\adec)$.
\par Choose a $1$-cell $\sigma$ in the $1$-skeleton of $\sqfm
i+1,Q\adec.$ which joins the two components of $\partial Q\adec$. Note
that $\sigma$ can be used to choose an ``origin'' of the angles for a
polar coordinate system $(r,\theta)$ on $Q\adec$. Formally, we
identify $Q\adec\simeq[0,5^{-i-1}]\times S^1$ and on
$Q\adec\setminus\sigma$ we have polar coordinates
$(r,\theta):Q\adec\setminus\sigma\to[0,5^{-i-1}]\times (0,2\pi)$. Moreover, the set $\hat
Q\adec\setminus\sigma$ is determined by the condition $r\in[5^{-i-2}, 5^{-i-1}-5^{-i-2}]$.
\par Let
  $\tilde\pi:\tilde Q\to Q$ denote the double cover and note that on
  $\Sigma = \tilde\pi^{-1}(Q\adec\setminus\sigma)$ we get a polar coordinate
  system $(r,\theta):\Sigma\to[0,
  5^{-i-1}]\times[(0,4\pi)\setminus\{2\pi\}]$, and that the map
  $\tilde\pi$, in polar coordinates, assumes the form
  $\tilde\pi(r,\theta)=(r,\theta\mod 2\pi)$. In particular,
  $\tilde\pi^{-1}(\sigma)$ divides $\Sigma$ in two sheets: $\Sigma_+$
  where $\theta\in(2\pi, 4\pi)$, and $\Sigma_-$ where
  $\theta\in(0,2\pi)$. We let $\chi$ denote the characteristic
  function of $\Sigma_+$; the following observation is crucial in the
  following:
  \begin{description}
  \item[(ShSep)] If $p,q\in\Sigma$, $d_{\tilde Q}(p,q)\le 5^{-i-3}$
    and $\tilde\pi(p)$ and $\tilde\pi(q)$ are on opposite sides of
    $\sigma$ (i.e~$|\theta(\tilde\pi(p)) -
    \theta(\tilde\pi(q))|\ge\pi$), then $\chi(p)\ne\chi(q)$.
  \end{description}
  \par We now define two helper functions $h_1,h_2:[0,4\pi]\to\real$:
  \begin{align}
    \label{eq:helperH1}
    h_1(\theta) &= \frac{\delta}{2\pi}\left(2\pi - |\theta -
      2\pi|\right),\\
    \label{eq:helperH2}
    h_2(\theta) &=
    \begin{cases}
      -\frac{\delta}{\pi}\theta&\text{if $\theta\in[0,\pi]$,}\\
      -\delta + \frac{\delta}{\pi}(\theta - \pi)&\text{if
        $\theta\in[\pi,3\pi]$,}\\
      \delta - \frac{\delta}{\pi}(\theta - 3\pi)&\text{if
        $\theta\in[3\pi, 4\pi]$.}
    \end{cases}
  \end{align}
  Note that the global Lipschtiz constants of $h_1$ and $h_2$ are:
  $\glip h_1. = \delta/(2\pi)$ and $\glip h_2. = \delta/\pi$. One also
  has the lower bound:
  \begin{equation}
    \label{eq:HlowBound}
    \inf_{\theta\in[0,2\pi]}\left[
      (h_1(\theta)-h_1(\theta+2\pi))^2 + (h_2(\theta) - h_2(\theta+2\pi))^2
      \right]^{1/2}\ge\frac{\delta}{2},
  \end{equation}
  which is proven in three cases; case $\theta\in[0,\pi/2]$: then
  $h_1(\theta)\le\delta/4$ and $h_1(\theta+2\pi)\ge 3\delta/4$; case
  $\theta\in[\pi/2, 3\pi/2]$: then $h_2(\theta)\in[-\delta,-\delta/2]$
  and $h_2(\theta+\pi)\in[\delta/2,\delta]$; case
  $\theta\in[3\pi/2,2\pi]$: then $h_1(\theta)\ge 3\delta/4$ and
  $h_1(\theta+2\pi)\le\delta/4$.
  \par We now define the $5$-Lipschitz cut-off function
  $\phi:[0,5^{-i-1}]\to\real$:
  \begin{equation}
    \label{eq:phiCutOff}
    \phi(r) = 
    \begin{cases}
      5r&\text{if $r\in[0, 5^{-i-2}]$,}\\
      5^{-i-1}&\text{if $r\in[5^{-i-2}, 5^{-i-1}-5^{-i-2}]$,}\\
      5^{-i-1}[1-5^{i+2}(r- 5^{-i-1}+5^{-i-2})]&\text{if
        $r\in[5^{-i-1}-5^{-i-2}, 5^{-i-1}]$,}
    \end{cases}
  \end{equation}
  and note that $\|\phi\|_\infty\le 5^{-i-1}$.
  \par 
  We now define $\Psi$ using polar coordinates:
  \begin{equation}
    \label{eq:PhiPolar}
    \begin{aligned}
      \Psi&:\Sigma\to\real^2\\
      (r,\theta)&\mapsto(\phi(r)h_1(\theta),\phi(r)h_2(\theta)),
    \end{aligned}
  \end{equation}
  and find the unique continuous extension $\Psi:\tilde Q\to\real^2$
  with $\Psi=0$ on $Q\cdec\cup Q\odec$. We now collect the important
  properties of $\Psi$. First, if $p_1,p_2\in\tilde \pi^{-1}(q)$ for
  $q\in Q\adec\setminus\sigma$ and $|\theta(p_1)-\theta(p_2)|=\pi$,
  then~(\ref{eq:HlowBound}) implies:
  \begin{equation}
    \label{eq:constr:basic_map_1}
    \|\Psi(p_1) - \Psi(p_2)\|_{\real^2}\ge\frac{\delta}{2}\phi(r(p_1)).
  \end{equation}
  Second from the upper bound on $\phi$ we get:
  \begin{equation}
    \label{eq:constr:basic_map_2}
    \|\Psi\|_{\real^2}\le\delta\diam Q,
  \end{equation}
 and third, from computing $d\Psi$ and using the standard Riemannian metric
 $r^2d\theta^2+dr^2$ on $\Sigma$, we estimate the global Lipschitz
 constant of $\Psi$:
  \begin{equation}
    \label{eq:constr:basic_map_3}
    \glip \Psi.\in[\delta,7\delta].
  \end{equation}
\end{constr}
In the following we let $\{e_i\}_{i=1}^\infty$ denote the standard orthonormal basis of $l^2$.
\begin{constr}[Construction of maps  $F_i:X_i\to l^2$]
  \label{constr:maps_to_hilbert}
  The map $F_0:X_0\to l^2$ is just an isometric embedding of the
  square $X_0$ in the plane $e_1\oplus e_2$. To get $F_1:X_1\to l^2$
  we modify $F_0\circ\pi_{1,0}$ by adding to it $\Psi_{\delta_1}\otimes (e_3\oplus e_4)$:
  this notation means that we take the map $\Psi$ from
  Construction~\ref{constr:basic_map} with $\delta=\delta_1$ and with $\tilde Q$
  the unique square $\{Q\}=\sqfm 0,X_0.$, 
  and then we identify the codomain of $\Psi$ with the plane $e_3\oplus
  e_4$. In particular note that:
  \begin{align}
    \label{eq:maps_to_hilbert_1}
    \|F_0\circ\pi_{1,0}-F_1\|_{\infty}&\lesssim \delta_15^{-1}
\\
    \label{eq:maps_to_hilbert_2}
    \glip F_1.&\lesssim(1+\delta_1^2)^{1/2}.
  \end{align}
  \par For $i\ge 1$, the map $F_{i+1}$ is defined by induction. We first
  have that $\im F_i$ is a subset of the hyperplane of $l^2$ spanned
  by the vectors $\{e_\alpha\}_{1\le \alpha\le 2i+2}$; then for each
  $Q\in\sqfm i,X_i.$ we choose $\Psi_{\delta_i,Q}:\tilde Q\to\real^2$
  as in Construction~\ref{constr:basic_map} setting
  $\delta=\delta_{i+1}$, and we then let:
  \begin{equation}
    \label{eq:constr:maps_to_hilbert_3}
    F_{i+1} = F_i\circ\pi_{i+1,i} + \sum_{Q\in\sqfm
      i,X_i.}\Psi_{\delta_i,Q}\otimes(e_{2i+3}\oplus e_{2i+4}).
  \end{equation}
  As we have inserted the new contributions in a plane orthogonal to
  $\im F_i$ we conclude that:
  \begin{equation}
    \label{eq:constr:maps_to_hilbert_4}
    \glip F_{i+1}.\lesssim(1+\delta_1^2+\cdots+\delta_i^2)^{1/2},
  \end{equation}
  and moreover:
  \begin{equation}
    \label{eq:constr:maps_to_hilbert_5}
    \|F_i\circ\pi_{i+1,i}-F_{i+1}\|_\infty\lesssim\delta_i5^{-i-1}.
  \end{equation}
\end{constr}
\begin{lem}[Convergence of the maps $F_i\circ\pi_{\infty,i}$]
  \label{lem:conv_mps_hilbert}
  The pull-backs $F_i\circ\pi_{\infty,i}$ converge uniformly to a map
  $F_\infty : X_\infty\to l^2$ whose Lipschitz constant satisfies:
  \begin{equation}
    \label{eq:conv_mps_hilbert_1}
    \glip F_\infty.\lesssim\left(1+\sum_i\delta_i^2\right)^{1/2}.
  \end{equation}
  Let $P_i:l^2\to l^2$ denote the orthogonal projection of $l^2$ onto
  the hyperplane spanned by $\{e_1,e_2,\cdots,e_{2i+1},
  e_{2i+2}\}$ and let $i\le j$ where $j=\infty$ is
  admissible. Defining $Y_j=F_j(X_j)$ we have a commutative diagram:
  \begin{equation}
    \label{eq:conv_mps_hilbert_2}
    \xy
    (0,20)*+{X_j} = "xj"; (20,20)*+{Y_j} = "yj";
    (0,0)*+{X_i} = "xi"; (20,0)*+{Y_i} = "yi";
    {\ar "xj"; "yj"}?*!/_2mm/{F_j};
    {\ar "xi"; "yi"}?*!/^2mm/{F_i};
    {\ar "xj"; "xi"}?*!/^4mm/{\pi_{j,i}};
    {\ar "yj"; "yi"}?*!/_4mm/{P_i};
    \endxy
  \end{equation}
\end{lem}
\begin{proof}
  By~(\ref{eq:constr:maps_to_hilbert_5}) the $F_i\circ\pi_{i+1,i}$
  converge uniformly and the limit map
  $F_\infty$ satisfies the Lipschitz
  bound~(\ref{eq:conv_mps_hilbert_1})
  as~(\ref{eq:constr:maps_to_hilbert_4}) implies a uniform bound on
  the Lipschitz constants of the $\{F_i\}_i$. When $j<\infty$ the
  commutativity of the diagram~(\ref{eq:conv_mps_hilbert_2}) follows
  from the definition of the maps $\{F_i\}_i$; for $j=\infty$ one passes
  the commutativity to the limit.
\end{proof}
In the following we let $N_\infty$ be the $2$-normal current
canonically associated to $X_\infty$: details and the precise definition of
$N_\infty$ are in \cite[Sec.~3]{schioppa_metric_sponge}. Recall also that,
even though $N_\infty$ is a \emph{metric current}, the calculus on
$X_\infty$ is similar to the classical one in $\real^2$, and $N_\infty$ admits a
``classical'' $2$-vector-field
representation: $N_\infty = \partial_x\wedge\partial_y\,d\mu_{X_\infty}$.
\begin{lem}[Existence and nontriviality of the $2$-current]
  \label{lem:nontriv_curr_hilbert}
  The push-forward $F_{\infty\#}N_\infty$ is a
  nontrivial $2$-normal current in $l^2$ supported on $Y_\infty$.
\end{lem}
\begin{proof}
  As $F_\infty$ is Lipschitz (actually it is a Lipschitz embedding,
  but \emph{not} 
  biLipschitz as the biLipschitz constants of the $F_i$ degrade as
  $i\nearrow\infty$), we only have to show that
  $F_{\infty\#}N_\infty$ is nontrivial. Let $x,y$ denote the standard
  ``coordinate'' functions on $e_1\oplus e_2$, and assume that $Y_0$ is
  normalized to be a unit square in that plane. Using the
  commutativity of the diagram~(\ref{eq:conv_mps_hilbert_2}) for
  $j=\infty$ and $i=0$ we get:
  \begin{equation}
    \label{eq:nontriv_curr_hilbert_p1}
    P_{0\#}F_{\infty\#}N_\infty(dx\wedge dy) = (F_0\circ
    \pi_{\infty,0})_{\#}N_\infty(dx \wedge dy) = F_{0\#}N_0(dx\wedge dy)=1,
  \end{equation}
  where $N_0$ denotes the current associated to $X_0$, i.e.~the
  anticlockwise-oriented unit square
  with the Lebesgue measure.
\end{proof}
\def\holes{\text{\normalfont Holes}}
\def\jonedec{^{(j-1)}}
\def\jdec{^{(j)}}
\def\ndec{^{(n)}}
\def\nonedec{^{(n-1)}}
\def\zerodec{^{(0)}}
\def\opoldec{\text{\normalfont o}}
\def\cpoldec{\text{\normalfont c}}
\def\apoldec{\text{\normalfont a}}
\begin{thm}[$2$-unrectifiability of $Y_\infty$]
  \label{thm:unrect_hilbert}
  $Y_\infty$ is purely $2$-unrectifiable in the sense that whenever
  $K\subset\real^2$ is compact and $\Phi:K\to l^2$ is Lipschitz,
  $\hmeas 2.(\Phi^{-1}(Y_\infty)\cap K)=0$.
\end{thm}
\begin{proof}
  We will argue by contradiction assuming that $K\subset
  \Phi^{-1}(Y_\infty)$ and that $\hmeas
  2.(K)>0$.
  \par\noindent\texttt{Step 1: Reduction to the case in which $\Phi$
    is a graph over $Y_0$.}
  \par Let $\Phi_n=P_n\circ \Phi$ and, using the Radon-Nikodym
  property of $l^2$, note that at each point $p\in K$
  of differentiability of $\Phi$ one has that each $\Phi_n$ is also
  differentiable at $p$ and that:
  \begin{equation}
    \label{eq:unrect_hilbert_p1}
    \lim_{n\to\infty}d\Phi_n(p) = d\Phi(p),
  \end{equation}
  where the limit is in the norm-topology of linear maps $\real^2\to
  l^2$. Following the notation of \cite[Sec.~4\&5]{ambrosio-rectifiability}, we let $J_2$ denote the Jacobian
  appearing in the area formula; by dominated convergence we then have:
  \begin{equation}
    \label{eq:unrect_hilbert_p2}
    \lim_{n\to\infty}\int_K\chi_E J_2(d\Phi_n)\,d\hmeas 2. = \int_K\chi_E J_2(d\Phi)\,d\hmeas 2.
  \end{equation}
  whenever $E\subset K$ is a Borel set. 
  \par We now consider the Borel set $E\subset K$
  consisting of those points which are Lebesgue density points of the
  set of points where $\Phi$ is differentiable and where $d\Phi_0$ has
  rank $<2$, and our goal is to show that:
  \begin{equation}
    \label{eq:unrect_hilbert_p2bis}
    \hmeas 2.(\im \Phi \cap P^{-1}_{0}(\Phi_0(E)))=0.
  \end{equation}
Note that the area formula \cite[Thm.~5.1]{ambrosio-rectifiability}
gives $\hmeas 2.(\Phi_0(E)\cap Y_0) =
0$. For each $n\ge 1$, using the square complex structure of
$\{X_i\}_{i\le n}$, the set $Y_n$ can be partitioned into finitely
many closed sets $\{S_\alpha\}_\alpha$ such that each restriction
$P_0|_{S_\alpha}:S_{\alpha}\to P_0(S_\alpha)$ is biLipschitz, thus
giving:
\begin{equation}
  \label{eq:unrect_hilbert_p3}
  \hmeas 2.(\Phi_n(E)\cap Y_n)=0.
\end{equation}
In particular, the area formula implies that:
\begin{equation}
  \label{eq:unrect_hilbert_p4}
  \int_K\chi_EJ_2(d\Phi_n)\,d\hmeas 2.=0. 
\end{equation}
Therefore, by~(\ref{eq:unrect_hilbert_p2bis}) we conclude that:
\begin{equation}
  \label{eq:unrect_hilbert_p5}
  \int_K\chi_EJ_2(d\Phi)\,d\hmeas 2.=0
\end{equation}
and then~(\ref{eq:unrect_hilbert_p2bis}) follows from the area formula.
\par Therefore by~(\ref{eq:unrect_hilbert_p2bis}) we can assume that
$d\Phi_0$ has full rank $2$ on the set of Lebesgue density points of
the set of differentiability points of $\Phi$. Using
\cite[Thm.~9]{kirchheim_metric_diff}, 
which is essentially a 
Lipschitz version of the Inverse Function Theorem, up to further partitioning $K$
and throwing away a set of null measure, we can assume that $\Phi$ is
$C$-biLipschitz and that $\Phi_0\circ\Phi = \text{Id}_{\Phi_0(K)}$. In
particular, we can assume that $K\subset Y_0$ and that $\Phi_0$ is just the
identity map.
\par\noindent\texttt{Step 2: Existence of square holes at scale $5^{-n}$.}
\par Note that the square-complex structure of $X_n$ induces a 
square-complex structure on $Y_n$ via the homeomorphism $F_n$; 
in the following for $i\ge n$ we will implicitly identify $\sqfm i,Y_n.$
with $\sqfm i, X_n.$.
\par Fix now $n$ and a square $Q\in\sqfm n-1,Y_0.$. Let $\hat Q\adec$
and $\sigma$ be as in Construction~\ref{constr:basic_map} and recall
that $\hat Q\adec$ consists of squares of $\sqfm n+2,Y_0.$.
\par We now fix a small parameter $c$ to be determined later in function of
the bi-Lipschitz constant $C$ of $\Phi$ and the Lipschitz constant of $F_\infty$. Let
\begin{equation}
  \label{eq:unrect_hilbert_p6}
  i_n = \lceil -\log_{5}(5^{-n-2}c\delta_n)\rceil
\end{equation}
and partition $\hat Q\adec$ into $\approx 5^{i_n-n}$ annuli
consisting of squares of $\sqfm i_n,Y_0.$. We consider one such an
annulus $A$. Our goal is to show that $K$ has to miss the interior of one of
the squares in $A$. 
\par We first order the squares $\{R_\alpha\}_{1\le\alpha\le t}$ of $A$ anticlockwise
so that $R_{\alpha+1}$ follows
$R_\alpha$, and $R_1$ follows $R_t$, and $R_1$ and $R_t$ meet along a
subsegment of $\sigma$. Assume that $K$ intersects each $\Inside(R_\alpha)$
and let $p_\alpha\in K\cap \Inside(R_\alpha)$. 
\par We first show that for $0\le j\le n$ the points
$\Phi_j(p_\alpha)$ and $\Phi_j(p_{(\alpha+1)\mod t})$ belong to the
same square of $\sqfm j,Y_j.$. In the following we use $\beta$ to
denote $\alpha$ or $(\alpha+1)\mod t$ and we will just write
$\alpha+1$ for $(\alpha+1)\mod t$.
\par For $j=0$ by construction $\Phi_0(p_\alpha)$ and
$\Phi_0(p_{\alpha+1})$ belong to the same square of $\sqfm 0,Y_0.$,
and for $j\ge 1$ we assume by induction that $\Phi_{j-1}(p_\alpha)$,
$\Phi_{j-1}(p_{\alpha+1})$ belong to the same $Q\jonedec_{j-1}\in\sqfm
j-1,Y_{j-1}.$. Let $Q\jdec_{j,\beta}\in\sqfm j,Y_j.$ denote the
square containing $\Phi_j(p_\beta)$ and assume by contradiction that
$Q\jdec_{j,\alpha}\ne Q\jdec_{j,\alpha+1}$. In the following we will
use the decorators ${}\adec$, ${}\odec$, ${}\cdec$ and $\hat{}$\ as in
Construction~\ref{constr:basic_map}: for example $\hat
Q_{j-1,\apoldec}\jonedec$ is obtained as $\hat Q\adec$ if we let
$Q=Q\jonedec_{j-1}$. In particular, as $Q\jdec_{j,\alpha}\ne
Q\jdec_{j,\alpha+1}$ we must have $P_{j-1}(Q\jdec_{j,\beta})\subset
Q\jonedec_{j-1,\apoldec}$. Let now $Q\zerodec_{i_n,\beta}\in\sqfm
 i_n,Q\zerodec_{n-1}.$ denote the square containing $p_\beta$, let
 $q\zerodec_\beta$ be its center, set
 $Q\jdec_{i_n,\beta}=P_0^{-1}
(Q\zerodec_{i_n,\beta})\cap
Q\jdec_{j,\alpha}$ and let $q\jdec_\beta$ denote its center.
\par As $\Phi$ is $C$-Lipschitz,
\begin{equation}
  \label{eq:unrect_hilbert_p7}
  d(\Phi_j(p_\alpha), \Phi_j(p_{\alpha+1})) \le 4C\times c5^{-n}\delta_n;
\end{equation}
as $F_j$ is $\glip F_\infty.$-Lipschitz,
\begin{equation}
  \label{eq:unrect_hilbert_p8}
  d(q_\beta\jdec, \Phi_j(p_\beta))\le 2\glip F_\infty.\times c5^{-n}\delta_n,
\end{equation}
so that:
\begin{equation}
  \label{eq:unrect_hilbert_p9}
  d(q\jdec_\alpha, q\jdec_{\alpha+1}) \le 4(C+\glip F_\infty.)\times c5^{-n}\delta_n.
\end{equation}
Let $S\jonedec_{j-1}=F^{-1}_{j-1}(Q\jonedec_{j-1})$ and
$S\jdec_{j,\beta}=F^{-1}_{j}(Q\jdec_{j,\beta})$; we must have
$S\jdec_{j,\alpha}\ne S\jdec_{j,\alpha+1}$ and
$\pi_{j-1}(F_j^{-1}(q_\beta\jdec))\in S\jonedec_{j-1,\apoldec}$. Note that
$F^{-1}_j(q\jdec_\beta)$ must be at distance $\ge 5^{-n}$ from $\partial
S\jonedec_{j-1,\apoldec}$ if $j\le n-1$ and at distance $\ge 5^{-n-3}$
if $j=n$ (in this case we use that $p_\beta\in \hat
Q_{n,\apoldec}\zerodec$), so that:
\begin{equation}
  \label{eq:unrect_hilbert_p10}
  \phi(r(F_j^{-1}(q\jdec_\beta)))\ge 5^{-n-3}.
\end{equation}
As $F_j^{-1}(q\jdec_\alpha)\ne F_j^{-1}(q\jdec_{\alpha+1})$, they belong to
different sheets of the double cover, and as
$\pi_{j-1}(S\jdec_{j,\alpha})$ and $\pi_{j-1}(S\jdec_{j,\alpha +1})$
are adjacent, we let $\hat q\jdec_\alpha$ be the center of the square of
$\sqfm i_n,Y_j.$ adjacent to $Q\jdec_{i_n,\alpha+1}$ and such that
$\pi_{j-1}(F^{-1}_j(\hat q\jdec_\alpha)) =
\pi_{j-1}(F_j^{-1}(q\jdec_\alpha))$. We now have:
\begin{align}
  \label{eq:unrect_hilbert_p11}
    r(F_j^{-1}(q\jdec_\alpha)) &= r(F_j^{-1}(\hat q\jdec_\alpha))\\
  \label{eq:unrect_hilbert_p12}
    \left|\theta(F_j^{-1}(q\jdec_\alpha)) - \theta(F_j^{-1}(\hat
      q\jdec_\alpha))\right| &= \pi,
\end{align}
and invoking~(\ref{eq:constr:basic_map_1}) we get:
\begin{equation}
  \label{eq:unrect_hilbert_p13}
  d(q\jdec_\alpha,\hat q\jdec_\alpha)\ge\frac{5^{-n-3}}{2}\delta_j\ge\frac{5^{-n-3}}{2}\delta_n.
\end{equation}
But as $\hat q\jdec_\alpha$ is the center of the  square of
$\sqfm i_n,Y_j.$ adjacent to $Q\jdec_{i_n,\alpha+1}$,
from~(\ref{eq:unrect_hilbert_p9}) we get:
\begin{equation}
  \label{eq:unrect_hilbert_p14}
  d(q\jdec_\alpha,\hat q\jdec_\alpha)\le 8(C+\glip F_\infty.)\times c5^{-n}\delta_n.
\end{equation}
Now, combining~(\ref{eq:unrect_hilbert_p14})
and~(\ref{eq:unrect_hilbert_p13}) and choosing $c\le 10^{-6}/(C+\glip
F_\infty.)$ we get a contradiction and conclude that
$Q\jdec_{j,\alpha} = Q\jdec_{j,\alpha+1}$.
\par A consequence of the previous discussion, specialized to $j=n$,
is that $\Phi_n(p_1)$ and $\Phi_n(p_t)$ belong to
the same sheet of the double cover $P_{n-1}^{-1}(\hat
Q_{n-1,\apoldec}\nonedec)\cap Y_n\to
\hat Q_{n-1,\apoldec}\nonedec$, while the choice of $c$ gives:
\begin{equation}
  \label{eq:unrect_hilbert_p15}
  d(F_n^{-1}(\Phi_n(p_1)), F_n^{-1}(\Phi_n(p_t)))\le 5^{-n-3},
\end{equation}
which contradicts~\textbf{(ShSep)}.
\par Let now $R_A$ denote the interior of the/a square of $A$ that $K$ misses. For each of
the $\approx 5^{i_n-n}$ annuli we can find such a square and
group them in a set $\holes(Q\zerodec_{n-1})$, and we have that:
\begin{equation}
  \label{eq:unrect_hilbert_p16}
  \hmeas 2.\left(\holes(Q\zerodec_{n-1})\right)\ge
  \gamma\delta_n\hmeas 2.(Q\zerodec_{n-1})
\end{equation}
for a constant $\gamma>0$ which does not depend on $n$ or
$Q\zerodec_{n-1}$. We thus conclude that 
\begin{equation}
  \label{eq:unrect_hilbert_p17}
  \hmeas 2.(K) \le\hmeas 2.\left(
    Y_0\setminus\bigcup_{Q\zerodec_{n-1}\in\sqfm n-1,Y_0.}\holes(Q\zerodec_{n-1})
    \right)\le (1-\gamma\delta_n)\hmeas 2.(Y_0).
\end{equation}
\par\noindent\texttt{Step 3: Cumulating the effects of holes and the choice
  of $\{\delta_n\}$.}
\par Let $Q_0$ denote the unique square of $\sqfm 0,Y_0.$. By
\texttt{Step 2} we have:
\begin{equation}
  \label{eq:unrect_hilbert_p18}
  \hmeas 2.(K) \le\hmeas 2.(Y_0\setminus\bigcup\holes(Q_0))\le
  (1-\gamma\delta_1)\hmeas 2.(Y_0).
\end{equation}
Now $\holes(Q_0)$ consists of squares of generation 
$<k_2=1+\lfloor G\log(1/\delta_1)\rfloor$ where $G$ is an appropriate
constant which depends on $c$ and $C$. As squares are nested,
if we apply \texttt{Step 2} on each of the squares of $\sqfm k_2,Y_0.$
which do not intersect the interior of $\bigcup\holes(Q_0)$ we get:
\begin{equation}
  \label{eq:unrect_hilbert_p19}
  \hmeas 2.(K)\le(1-\gamma\delta_1)(1-\gamma\delta_{k_2}).
\end{equation}
In general, we can reiterate, and get:
\begin{equation}
  \label{eq:unrect_hilbert_p20}
  \hmeas 2.(K)\le\prod_j(1-\gamma\delta_{k_j}),
\end{equation}
where $k_1=1$ and $k_{j+1}=k_j+\lfloor G\log(1/\delta_i)\rfloor$. If
we had 
\begin{equation}
    \label{eq:unrect_hilbert_p21}
\sum_j\delta_{k_j}=\infty
\end{equation}
we would finish obtaining the
contradiction $\hmeas 2.(K)=0$.
\par We show that~(\ref{eq:unrect_hilbert_p21}) holds if
$\delta_n=\frac{1}{10+n}$. For simplicity we assume that logarithms
are in base $10$. We use the estimate:
\begin{equation}
  \label{eq:unrect_hilbert_p22}
  \sum_{j=10^t}^{10^{t+1}}\frac{1}{j}\ge\frac{\log 10^{t+1}-\log 10^t}{16}=\frac{1}{16}.
\end{equation}
If $k_i\in(10^t,10^{t+1})$ then $k_i$ and $k_{i+1}$ are separated by a
distance $\le 23(t+1)$. Hence we have:
\begin{equation}
  \label{eq:unrect_hilbert_p23}
  \sum_{10^{t}\le k_i<10^{t+1}}\delta_{k_i}\ge\frac{1}{42(t+1)}.
\end{equation}
We thus have:
\begin{equation}
  \label{eq:unrect_hilbert_p24}
  \sum_j\delta_{k_j}\ge\lim_{T\to\infty}\sum_{t=2}^T\frac{1}{42(t+1)}=\infty.
\end{equation}
\end{proof}
\section{$2$-current in $\real^4$}
\label{sec:2r4}
In $\real^4$ we have both to construct the metric spaces $X_n$ and the
embeddings as the construction cannot be self-similar.
\def\Ndec{^{(N)}}
\def\nwsqfm#1,#2.{\setbox1=\hbox{$#1$\unskip}\setbox2=\hbox{$#2$\unskip}\text{\normalfont
    Sq}_{\ifdim\wd1>0pt #1\fi}({\ifdim\wd2>0pt #2\else X_1\fi}) }
\def\dirset#1.{\setbox1=\hbox{$#1$\unskip}\text{\normalfont
    Th}({\ifdim\wd1>0pt #1\else 1\fi})}
\def\dirnset#1.{\setbox1=\hbox{$#1$\unskip}\text{\normalfont
    Th}_0({\ifdim\wd1>0pt #1\else 1\fi})}
\def\RadN#1.{\setbox1=\hbox{$#1$\unskip}\text{\normalfont
    RN}({\ifdim\wd1>0pt #1\else 1\fi})}
\def\oskel#1.{\setbox1=\hbox{$#1$\unskip}\text{\normalfont
    Sk}_1({\ifdim\wd1>0pt #1\else X_1\fi})}
\def\addec{\text{\normalfont ad}}
\def\cmass#1.{{\|#1\|}}
\def\PAR{\text{\normalfont PAR}}
\begin{constr}[2-Normal current in $\real^4$]
  \label{constr:r4}
\par\noindent\texttt{Step 1: Affine approximation of $\Psi_\delta$.}
\par Let $Q$, $\tilde Q$, $\Psi_\delta$, etc...~be as in
Construction~\ref{constr:basic_map}. Thee maps $h_1$, $h_2$ and $\phi$
are piecewise-affine, while $\theta$ and $r$, which are defined on
$\Sigma$, are not so. However, by taking iterated subdivisions of $Q$
and $\tilde Q$, we can approximate $\theta$ and $r$ by maps which are
affine on each square of $\Sigma\Ndec$; letting $N\to \infty$ one can
take the approximations as close as one wants in the uniform topology
while keeping the Lipschitz constants bounded. Thus, there are an
$N\in\natural$, independent of $\delta$, and a piecewise-affine map
\begin{equation}
  \label{eq:constr_r4_1}
  \Phi_\delta:\tilde Q\Ndec\to\real^2
\end{equation}
such that the corresponding of~(\ref{eq:constr:basic_map_1}),
(\ref{eq:constr:basic_map_2}) and (\ref{eq:constr:basic_map_3}) hold:
 \begin{align}
   \label{eq:constr_r4_3}
   \glip
   \Phi_\delta.&\in\bigl[\frac{\delta}{16},23{\delta}\bigr]\\
   \label{eq:constr_r4_4}
   \|\Phi_\delta(p_1)-\Phi_\delta(p_2)\|_{\real^2}&\ge\frac{\delta}{3}\phi(r(p_1))\\
   \label{eq:constr_r4_5}
   \|\Phi_\delta\|_{\real^2}&\le 2\delta\diam Q.
 \end{align}
\par\noindent\texttt{Step 2: Construction of $F_1$.}
\par Let $X_0=[0,1]^2$ and $F_0:X_0\to e_1\oplus e_2\subset\real^4$ be
the standard isometric embedding; $X_1$ is obtained by applying to
$X_0$ Construction~\ref{constr:basic_map} as in the $l^2$-case and
then we let:
\begin{equation}
  \label{eq:constr_r4_6}
  F_1 = F_0\circ\pi_{1,0}+\Phi_{\delta_1}\otimes(e_3\oplus e_4).
\end{equation}
Note that we have bounds on the Lipschitz constant of $F_1$:
\begin{equation}
  \label{eq:constr_r4_7}
  \glip F_1.\in\bigl[\frac{(1+\delta_1^2)^{1/2}}{16}, 23(1+\delta_1^2)^{1/2}\bigr]
\end{equation}
and that because of~(\ref{eq:constr_r4_4}) $F_1$ is a topological
embedding, being injective. Let $\nwsqfm ,.$ denote the set of squares of $X_1$ and let
$Y_1=F_1(X_1)$. As $F_1$ is piecewise affine, each $Q\in\nwsqfm ,.$
determines a unique affine $2$-plane $\tau(Q)\subset\real^4$ which
 contains $F_1(Q)$; the corresponding unique $2$-plane parallel to
 $\tau(Q)$ and passing through the origin will be denoted by
 $\tau_0(Q)$; we finally let:
 \begin{align}
   \label{eq:constr_r4_8}
   \dirset .&= \bigcup_{Q\in \nwsqfm ,.}\tau(Q)\\
   \label{eq:constr_r4_9}
   \dirnset .&= \bigcup_{Q\in \nwsqfm ,.}\tau_0(Q),
 \end{align}
 and note that both sets are finite.
\par\noindent\texttt{Step 3: The Radial Basis Neighbourhood.}
\par For $Q\in\nwsqfm ,.$ we let $\pi_{\tau(Q)}$ denote the orthogonal
projection onto $\tau(Q)$ and define the \textbf{radial-basis function}:
\begin{equation}
  \label{eq:constr_r4_10}
  \varphi_Q(x) = 
  \begin{cases}
    \exp\left(
      -\frac{\sigma_1}{\dist(\pi_{\tau(Q)}(x), F_1(\partial Q))}
    \right)\times 46\diam(F_1(Q))&\text{if
      $\pi_{\tau(Q)}\in\Inside(F_1(Q))$}\\
    0&\text{otherwise,}
  \end{cases}
\end{equation}
where $\sigma_1>0$ is a parameter to be chosen later. We then define
the \textbf{radial basis neighbourhood}:
\begin{multline}
  \label{eq:constr_r4_11}
  \RadN . = \biggl\{
  p\in\real^4:\text{there is a $Q\in\nwsqfm ,.$}: p = x+y, x\in
  F_1(Q),\\
  y\perp
  \tau(Q), \text{and $\|y\|\le\varphi_Q(x)$}
\biggr\}.
\end{multline}
$\RadN .$ is not a neighbourhood of $Y_1$ as about each point of
$F_1(\partial Q)$ it has empty interior; however, it is close to being a
neighbourhood of $Y_1$ as it contains a neigbhbourhood of:
\begin{equation}
  \label{eq:constr_r4_12}
  \bigcup_{Q\in\nwsqfm ,X_1.}\Inside(F_1(Q)).
\end{equation}
We define $P_1:\RadN .\to Y_1$ by $p=x+y\mapsto x$. Note that if
$\sigma_1$ is sufficiently large $P_1$ is well-defined
(see Lemma~\ref{lem:claim_j}), and that:
\begin{description}
\item[(Claim1)] For each $\varepsilon_1>0$ there is a $\sigma_1>0$
  such that $P_1$ is $(1+\varepsilon_1)$-Lipschitz. 
\end{description}
\par\noindent\texttt{Step 4: The adaptative subdivision of $X_1$ and
  the construction of $X_2$.}
\par Let $\oskel.$ denote the $1$-skeleton of $X_1$ (i.e.~the union of
$1$-and-$0$-dimensional cells) and $\nwsqfm k,X_1.$ the set of squares
obtained by subdividing the squares of $\nwsqfm ,.$ $k$-times (i.e.~we
get $5^{2k}$-isometric subsquares from each $Q\in\nwsqfm ,.$). Let
\begin{equation}
  \label{eq:constr_r4_r13}
  \nwsqfm \infty,. = \bigcup_{k\ge 1}\nwsqfm k,.;
\end{equation}
we say that $Q\in \nwsqfm \infty,.$ is \textbf{adapted} to $\RadN.$ if
the $(23\delta_1\diam F_1(Q))$-neighborhood of $F_1(Q)$ is contained
in $\RadN.$ and if, denoting by $\PAR(Q)\in\nwsqfm ,.$ the unique
square containing $Q$, one has:
\begin{equation}
  \label{eq:tritanomaly-aug-add1}
  \max_{x\in Q}\dist(x,\partial Q) \le \delta_1\max_{x\in Q}\dist(x,\partial\PAR(Q)).
\end{equation}
Now the set of adapted squares is partially ordered by
inclusion and we let $\nwsqfm\addec,.$ denote the set of its maximal
elements. Note that the elements of $\nwsqfm\addec,.$ must have
pairwise disjoint interia and:
\begin{equation}
  \label{eq:constr_r4_r14}
  X_1\setminus\oskel . = \bigcup_{Q\in\nwsqfm\addec,.}Q.
\end{equation}
We obtain $X_2$ from $X_1$ by applying
Construction~\ref{constr:basic_map} to each $Q\in\nwsqfm\addec, X_1.$,
and subdividing the resulting squares $N$-times as in \texttt{Step 1}. Now
$X_2$ is not a square complex, but it is almost so. First, $X_2$ is
the limit on an admissible inverse system in the sense of
Definition~3.1 in~\cite{schioppa_metric_sponge}. As on $X_0$ and $X_1$ we
considered the canonical measures constructed in
Section~\ref{sec:2hilb}, we obtain a canonical measure $\mu_2$ on
$X_2$ so that $(X_2,\mu_2)$ is a $(1,1)$-PI space (see Theorem~3.8 in
\cite{schioppa_metric_sponge}). As the metric on $X_2$ we will
consider the length metric and we observe that $X_2$ is doubling with
doubling constant $\le 15$. We also obtain a $1$-Lipschitz map
$\pi_{2,1}:X_2\to X_1$ as the inverse limit system associated to $X_2$
is built on top of $X_1$. By Theorem~3.20 in
\cite{schioppa_metric_sponge} we obtain a $2$-dimensional simple
normal current $N_2$ with $\cmass N_2.=\mu_2$ and $\pi_{2,1\#}N_2 =
N_1$, $N_1$ being the canonical normal current associated to $X_1$.
\par Second $\oskel .$ embedds isometrically in $X_2$ and, away from
$\oskel .$, $X_2$ has a square complex structure. In fact, each
$Q\in\nwsqfm\addec, X_1.$ gives rise to at most $10\times 5^{N+7}$
squares in $X_2$; we thus denote the set of such squares by $\nwsqfm,
X_2.$ and let:
\begin{equation}
   \label{eq:constr_r4_r15}
   \oskel X_2. = \oskel X_1.\cup\bigcup_{Q\in\nwsqfm,X_2.}\oskel Q..
\end{equation}
\par\noindent\texttt{Step 5: The construction of $F_2$.}
\par To get $X_2$ we have applied to each $Q\in\nwsqfm\addec, X_1.$
Construction~\ref{constr:basic_map} and we have further subdivided
$N$-times the squares of the branched cover $\tilde Q\to Q$ so that we
can define $\Phi_{Q,\delta_2}:\tilde Q\to\real^2$ as in \texttt{Step
  1}. However, we need a bit extra care to get finitely many
possibilities for the tangent space of $Y_2$: this will be useful in
the proof of Lemma~\ref{lem:claim_j}.
\par First, for $Q_1\ne Q_2\in\nwsqfm\addec ,X_1.$ the maps
$\Phi_{Q_1,\delta_2}$ and $\Phi_{Q_2,\delta_2}$ can be taken to be the same
up to composition with translations and dilations. Second, each
$Q\in\nwsqfm\addec, X_1.$ belongs to a unique parent
$\PAR(Q)\in\nwsqfm,.$. As $\dirnset .$ is finite, we can choose a
finite set of pairs $\{(e_{1,Q},e_{2,Q})\}_{Q\in\nwsqfm\addec, .}$
such that each pair $(e_{1,Q},e_{2,Q})$ is an orthonormal basis
of the $2$-plane orthogonal to $\tau_0(\PAR(Q))$. We let:
\begin{equation}
  \label{eq:constr_r4_r16}
  F_2(x) = F_1\circ\pi_{2,1}(x) +
  \sum_{Q\in\nwsqfm\addec,.}\Phi_{Q,\delta_2}(x)\otimes(e_{1,Q}\oplus e_{2,Q}),
\end{equation}
and observe that by~(\ref{eq:constr_r4_4}) $F_2$ is a topological
embedding. As $X_2$ is a length space and as $\Phi_{Q,\delta_2}$ adds
a contribution to the gradient of $F_1$ orthogonally to
$\tau_0(\PAR(Q))$, we get:
\begin{equation}
  \label{eq:constr_r4_r17}
  16^{-1}(1+\delta_1^2+\delta_2^2)^{1/2}\le\glip F_2.\le 23(1+\delta_1^2+\delta_2^2)^{1/2},
\end{equation}
and we also have:
\begin{equation}
  \label{eq:constr_r4_r18}
  \|F_1\circ\pi_{2,1}-F_2\|_\infty\le 56\times 5^{-2}\delta_2.
\end{equation}
\par Let $Y_2=F_2(X_2)$ and note that $F_2$ is affine when restricted
to each $Q\in\nwsqfm, X_2.$. We let $\tau(F_2(Q))$ denote the affine
$2$-plane containing $F_2(Q)$ and $\tau_0(F_2(Q))$ the corresponding
$2$-plane passing through the origin. We finally let
\begin{align}
   \label{eq:constr_r4_19}
   \dirset 2.&= \bigcup_{Q\in \nwsqfm ,X_2.}\tau(Q)\\
   \label{eq:constr_r4_20}
   \dirnset 2.&= \bigcup_{Q\in \nwsqfm ,X_2.}\tau_0(Q),
 \end{align}
and note that $\dirnset 2.$ is finite by the choice of
$\{(e_{1,Q},e_{2,Q})\}_{Q\in\nwsqfm\addec, .}$ (while $\dirset 2.$ is
not finite). By construction we also have the commutative diagram:
  \begin{equation}
    \label{eq:constr_r4_21}
    \xy
    (0,40)*+{X_2} = "x2"; (20,40)*+{Y_2} = "y2";
    (0,20)*+{X_1} = "x1"; (20,20)*+{Y_1} = "y1";
    (0,0)*+{X_0} = "x0"; (20,0)*+{Y_0} = "y0";
    {\ar "x2"; "y2"}?*!/_2mm/{F_2};
    {\ar "x1"; "y1"}?*!/_2mm/{F_1};
    {\ar "x0"; "y0"}?*!/^2mm/{F_0};
    {\ar "x2"; "x1"}?*!/^4mm/{\pi_{2,1}};
    {\ar "x1"; "x0"}?*!/^4mm/{\pi_{1,0}};
    {\ar "y2"; "y1"}?*!/_4mm/{P_1};
    {\ar "y1"; "y0"}?*!/_4mm/{P_0};
    \endxy
  \end{equation}
\par\noindent\texttt{Step 6: The general iteration.}
\par Assume we have constructed $\{X_k\}_{k\le j}$, $\{\RadN
k.\}_{k\le j-1}$ and $\{F_k\}_{k\le j}$; for $Q\in\nwsqfm, X_j.$ we
define the \textbf{radial basis function}:
\begin{equation}
  \label{eq:constr_r4_22}
  \varphi_Q(x) = 
  \begin{cases}
    \exp\left(
      -\frac{\sigma_j}{\dist(\pi_{\tau(Q)}(x), F_j(\partial Q))}
    \right)\times 46\diam(F_j(Q))&\text{if
      $\pi_{\tau(Q)}\in F_j(\Inside(Q))$}\\
    0&\text{otherwise,}
  \end{cases}
\end{equation}
where $\sigma_j>0$ is a parameter to be chosen later. We then define
the \textbf{radial basis neighbourhood}:
\begin{multline}
  \label{eq:constr_r4_23}
  \RadN j. = \biggl\{
  p\in\real^4:\text{there is a $Q\in\nwsqfm ,X_j.$}: p = x+y, x\in
  F_j(Q),\\
  y\perp
  \tau(Q), \text{and $\|y\|\le\varphi_Q(x)$}
\biggr\}.
\end{multline}
As for $\RadN 1.$, $\RadN j.$ is not a neighbourhood of $Y_j$ but it
is a neighbourhood of
\begin{equation}
  \label{eq:constr_r4_24}
  \bigcup_{Q\in\nwsqfm ,X_j.}\Inside(F_j(Q)).
\end{equation}
We define $P_j:\RadN j.\to Y_j$ by $p=x+y\mapsto x$ and will later
show that if
$\sigma_j$ is sufficiently large, $P_j$ is well-defined
(see Lemma~\ref{lem:claim_j}), and that:
\begin{description}
\item[(Claim $j$)] For each $\varepsilon_j>0$ there is a $\sigma_j>0$
  such that $P_j$ is $(1+\varepsilon_j)$-Lipschitz. 
\end{description}
We then define as above:
\begin{equation}
  \label{eq:constr_r4_r25}
  \nwsqfm \infty, X_j. = \bigcup_{k\ge 1}\nwsqfm k, X_j.;
\end{equation}
we say that $Q\in \nwsqfm \infty, X_j.$ is \textbf{adapted} to $\RadN j.$ if
the $(23\delta_j\diam F_j(Q))$-neighborhood of $F_j(Q)$ is contained
in $\RadN j.$ and if, denoting by $\PAR(Q)\in\nwsqfm ,.$ the unique
square containing $Q$, one has:
\begin{equation}
  \label{eq:tritanomaly-aug-add2}
  \max_{x\in Q}\dist(x,\partial Q) \le \delta_j\max_{x\in Q}\dist(x,\partial\PAR(Q)).
\end{equation} As above we let $\nwsqfm\addec,X_j.$ be the set of
maximal adapted squares, which must then have pairwise disjoint interia
and satisfy:
\begin{equation}
  \label{eq:constr_r4_r26}
  X_j\setminus\oskel X_j. = \bigcup_{Q\in\nwsqfm\addec,X_j.}Q.
\end{equation}
We obtain $X_{j+1}$ from $X_j$ by applying
Construction~\ref{constr:basic_map} to each $Q\in\nwsqfm\addec,X_j.$
and subdividing the obtained squares other $N$-times. As discussed
above, $X_{j+1}$ is not a square complex, but it is almost so. In
fact, $X_{j+1}$ is the limit of an admissible inverse system in the
sense of Definition~3.1 of \cite{schioppa_metric_sponge}. We get a
$1$-Lipschitz map $\pi_{j+1,j}:(X_{j+1},\mu_{j+1})\to(X_j,\mu_j)$ and
$X_{j+1}$ is a doubling length space with doubling constant
$\le 50$ (the projection of a square of $\nwsqfm\addec,X_j.$ contains at most 
$50$ squares of $1/5$-the side length). As in \texttt{Step 4} we find that to $X_{j+1}$ is
canonically associated a normal metric current $N_{j+1}$ with
$\pi_{j+1,j\#}N_{j+1}=N_j$ and $\cmass N_{j+1}.=\cmass N_j.$. We let
$\nwsqfm ,X_{j+1}.$ be the corresponding set of squares of $X_{j+1}$,
which has a square-complex structure away from:
\begin{equation}
  \label{eq:constr_r4_r27}
  \oskel X_{j+1}.=\bigcup_{k\le j}\oskel
  X_j.\cup\bigcup_{Q\in\nwsqfm,X_{j+1}.}\oskel Q.;
\end{equation}
note also that:
\begin{equation}
  \label{eq:constr_r4_r28}
  X_{j+1}\setminus\bigcup_{k\le j}\oskel X_j.=\bigcup_{Q\in\nwsqfm,X_{j+1}.}Q.
\end{equation}
To construct $F_{j+1}$ we proceed as for $F_2$: for
$Q\in\nwsqfm\addec,X_j.$ we choose $\Phi_{Q,\delta_{j+1}}:\tilde
  Q\to\real^2$ such that for $Q_1\ne Q_2$ the maps
  $\Phi_{Q_1,\delta_{j+1}}$ and $\Phi_{Q_2,\delta_{j+1}}$ can be taken
  to differ up to composition with translations and
  dilations. Secondly, each $Q\in\nwsqfm\addec,X_j.$ belongs to a
  unique parent $\PAR(Q)\in\nwsqfm,X_j.$ and $\dirnset j.$ is
  finite. Thus we can choose a finite set of pairs
  $\{(e_{1,Q},e_{2,Q})\}_{Q\in\nwsqfm\addec,X_j.}$ such that each
  $(e_{1,Q},e_{2,Q})$ is an orthonormal basis of the orthogonal
  complement of $\tau_0(\PAR(Q))$. We define:
\begin{equation}
  \label{eq:constr_r4_r29}
  F_{j+1}(x) = F_j\circ\pi_{j+1,j}(x) +
  \sum_{Q\in\nwsqfm\addec,X_j.}\Phi_{Q,\delta_{j+1}}(x)\otimes(e_{1,Q}\oplus e_{2,Q}),
\end{equation}
and observe that by~(\ref{eq:constr_r4_4}) $F_{j+1}$ is a topological
embedding. As $X_{j+1}$ is a length space and as $\Phi_{Q,\delta_{j+1}}$ adds
a contribution to the gradient of $F_j$ orthogonally to
$\tau_0(\PAR(Q))$, we get:
\begin{equation}
  \label{eq:constr_r4_r30}
  16^{-1}(1+\sum_{l=1}^{j+1}\delta_l^2)^{1/2}\le\glip F_{j+1}.\le 23(1+\sum_{l=1}^{j+1}\delta_l^2)^{1/2},
\end{equation}
and we also have:
\begin{equation}
  \label{eq:constr_r4_r31}
  \|F_{j}\circ\pi_{j+1,j}-F_{j+1}\|_\infty\le 56\times 5^{-j}\delta_{j+1}.
\end{equation}
Let $Y_{j+1}=F_{j+1}(X_{j+1})$ and note that $F_{j+1}$ is affine when
restricted to each $Q\in\nwsqfm ,X_{j+1}.$; as in \texttt{Step 5} we
define $\tau(F_{j+1}(Q))$, $\tau_0(F_{j+1}(Q))$, $\dirset j+1.$ and
$\dirnset j+1.$, and observe that $\dirnset j+1.$ is finite.
\par Finally for $j\le k$ one has the following commutative diagrams:
  \begin{equation}
    \label{eq:constr_r4_32}
    \xy
    (0,20)*+{X_{j+1}} = "xjone"; (20,20)*+{Y_{j+1}} = "yjone";
    (0,0)*+{X_k} = "xk"; (20,0)*+{Y_k} = "yk";
    {\ar "xjone"; "yjone"}?*!/_2mm/{F_{j+1}};
    {\ar "xk"; "yk"}?*!/_2mm/{F_k};
    {\ar "xjone"; "xk"}?*!/^7mm/{\pi_{j+1,k}};
    {\ar "yjone"; "yk"}?*!/_16mm/{P_k\circ P_{k-1}\circ\cdots\circ P_j};
    \endxy
  \end{equation}
\end{constr}
\begin{lem}[Convergence of the spaces and currents]
  \label{lem:convergence}
  The metric measure spaces $(X_n,\mu_n)$ converge in the mGH-sense to
  $(X_\infty,\mu_\infty)$; having arranged convergence in a container,
  the normal currents $N_n$ converge weakly to a normal current
  $N_\infty$ supported in $X_\infty$ with $\cmass
  N_\infty.=\mu_\infty$; the maps $\pi_{n,i}:X_n\to X_i$ also converge
  to $1$-Lipschitz maps $\pi_{\infty,i}:X_\infty\to X_i$ as
  $n\nearrow\infty$ and, for each pair $l<i$, one has commutative diagrams:
  \begin{equation}
    \label{eq:convergence_s1}
    \xy
    (0,20)*+{(X_{\infty},\mu_\infty, N_\infty)} = "xinf"; 
    (40,20)*+{(X_i,\mu_i,N_i)} = "xi";
    (0,0)*+{(X_l,\mu_l,N_l)} = "xl"; 
    {\ar "xinf"; "xi"}?*!/_2mm/{\pi_{\infty,i}};
    {\ar "xinf"; "xl"}?*!/^5mm/{\pi_{\infty,l}};
    {\ar "xi"; "xl"}?*!/_3mm/{\pi_{i,l}};
    \endxy
  \end{equation}
\end{lem}
\begin{proof}
  The proof is routine as $(X_\infty,\mu_\infty)$ is an inverse limit
  of the metric measure spaces $(X_k,\mu_k)$. Even though here we work
  with a slightly more general cube complexes (in $X_k$ we allow cells of
  of different diameters), the same arguments as in
  \cite[Sec.~3]{schioppa_metric_sponge} go through.
\end{proof}
\begin{lem}[Proof of (\textbf{Claim}$j$)]
  \label{lem:claim_j}
  If the $\delta_k$'s are chosen  so that:
  \begin{equation}
    \label{eq:claim_j_s1}
    4\cdot 10^3\left(1+\sum_{k\ge1}\delta_k^2\right)^{1/2}\left(\sum_{k\ge1}\delta_k^2\right)<\frac{1}{8},
  \end{equation}
  then {\normalfont(\textbf{Claim} $j$)} holds.
\end{lem}
\def\pardec{^\text{\normalfont (par) }}
\begin{proof}
  \par\noindent\texttt{Step 1: The case $j=1$.}
  \par As $\dirset .$ is finite and $F_1$ is an isometric embedding
  plus a small Lipschitz perturbation,
  we can find an $\alpha>0$ such that if $\{Q_1,Q_2\}\subset\nwsqfm,
  X_1.$ are distinct and $x_t\in F(Q_t)$ ($t=1,2$) then:
  \begin{equation}
    \label{eq:claim_j_p1}
    \|x_1-x_2\|\ge\alpha\max_{t=1,2}\dist(x_t, F_1(\partial Q_t)).
  \end{equation}
  Let $x_1+y_1,x_2+y_2\in\RadN .$; then
  \begin{equation}
    \label{eq:claim_j_p2}
    \|y_t\|\le c(\sigma_1)\dist(x_t, F_1(\partial Q_t)),
  \end{equation}
  where $\lim_{\sigma_1\to\infty}c(\sigma_1)=0$. Therefore,
  \begin{equation}
    \label{eq:claim_j_p3}
    \|(x_1+y_1) - (x_2+y_2)\|_2\ge\|x_1-x_2\|-c(\sigma_1)(\|y_1\|+\|y_2\|),
  \end{equation}
  from which we get:
  \begin{equation}
    \label{eq:claim_j_p4}
    \left(1+\frac{2}{\alpha}c(\sigma_1)\right)\|(x_1+y_1)-(x_2+y_2)\|\ge\|x_1-x_2\|=\|P_1(x_1+y_1)
    -P_1(x_2+y_2)\|.
  \end{equation}
  Choosing $\sigma_1$ sufficiently small we obtain that $P_1$ is
  well-defined and $(1+\varepsilon_1)$-Lipschitz (note that for the
  case in which $Q_1=Q_2$ we have $\alpha=1$
  in~(\ref{eq:claim_j_p4})).
  \par\noindent\texttt{Step 2: The case $j>1$.}
  \par By induction we assume the existence of $\eta>0$ such that if
  $k\le j-1$, $x_t\in F_k(Q_t)$ ($t=1,2$ and $Q_t\in\nwsqfm ,X_k.$)
  where $Q_1\ne Q_2$, then:
  \begin{equation}
    \label{eq:claim_j_p5}
    \|x_1-x_2\|\ge\eta\max_{t=1,2}\dist(x_t, F_k(\partial Q_t)).
  \end{equation}
  We want to establish an analogue of~(\ref{eq:claim_j_p4}), but we
  will need to consider $3$ possibilities; we define:
  \begin{equation}
    \label{eq:claim_j_p6}
    P_{i,k}=P_k\circ\cdots\circ P_{i-1}\circ P_i\quad(\text{compare~(\ref{eq:unif_comp_s1})}),
  \end{equation}
  and we let $Q_{k,t}$ denote the square of $\nwsqfm,X_k.$ containing
  $F_k^{-1}(P_{j-1,k}(x_t))$.
  \par First assume thar for some $k\le j-1$ $Q_{k,1}\ne Q_{k,2}$ and
  let $k_0$ be the minimal value of $k$ such that this happens. Then:
  \begin{equation}
    \label{eq:claim_j_p7}
    \|P_{j-1,k_0}(x_1)-P_{j-1,k_0}(x_2)\|\ge\eta\max_{t=1,2}\dist(
    P_{j-1,k_0}(x_t),F_{k_0}(\partial Q_{k_0,t})
    ).
  \end{equation}
  By induction we will assume that $P_{j-1,k_0}$ is well-defined with
  $\glip P_{j-1,k_0}.<\infty$. Let $q_t\in F_{k_0}(\partial
  Q_{k_0,t})$ be a closest point to $x_t$. As $F_{k_0}|Q_{k_0,t}$ is
  affine satisfying~(\ref{eq:constr_r4_r30}), we conclude that:
  \begin{equation}
    \label{eq:claim_j_p8}
    \frac{\|P_{j-1,k_0}(x_t)-q_t\|}{
      d(F_{k_0}^{-1}(P_{j-1,k_0}(x_t)), F_{k_0}^{-1}(q_t))
    }\in\left[\frac{(1+\sum_{k\le k_0}\delta_k^2)^{1/2}}{16},
      23(1+\sum_{k\le k_0}\delta_k^2)^{1/2}\right].
  \end{equation}
  For $k_0<k\le j-1$ let $Q_{k,t}\pardec$ denote the square of
  $\nwsqfm\addec,X_{k-1}.$ containing $\pi_{k,k-1}(Q_{k,t})$. From the
  definition of $F_k$ we get:
  \begin{equation}
    \label{eq:claim_j_p9}
    x_t - P_{j-1,k_0}(x_t) = \sum_{k_0+1\le k\le
      j}\Phi_{Q_{k,t}\pardec,\delta_k}(\pi_{j,k}\circ
    F_j^{-1}(x_t))\otimes (e_{1,Q_{k,t}\pardec}\oplus e_{2,Q_{k,t}\pardec}).
  \end{equation}
  From the bound on the Lipschitz constant of
  $\Phi_{Q_{k,t}\pardec,\delta_k}$ we get:
  \begin{equation}
    \label{eq:claim_j_p10}
    \|\Phi_{Q_{k,t}\pardec,\delta_k}(\pi_{j,k}\circ
    F_j^{-1}(x_t))\|_{\real^2}\le 28\delta_k \dist(\pi_{j,k}\circ
    F_j^{-1}(x_t), \partial Q_{k,t}\pardec);
  \end{equation}
  recall from \texttt{Step 6} in~\ref{constr:r4} that $\partial
  Q_{k_0,t}$ is isometrically embedded in $X_k$ for $k\ge k_0$; as
  geodesic paths joining a point $p\in X_k$ to a point $q\in\oskel
  X_k.$ can be taken not to pass through different sheets of the
  double covers and, minding~(\ref{eq:tritanomaly-aug-add2}), we have for $k_0<k\le j-1$:
  \begin{equation}
    \label{eq:claim_j_p11}
    \dist(\pi_{j,k}\circ
    F_j^{-1}(x_t), \partial Q_{k,t}\pardec)\le \delta_k\dist(\pi_{j,k_0}\circ
    F_j^{-1}(x_t), \partial Q_{k_0,t}).
  \end{equation}
  Commbining~(\ref{eq:claim_j_p9}),
  (\ref{eq:claim_j_p10}) and~(\ref{eq:claim_j_p11}) we get:
  \begin{equation}
    \label{eq:claim_j_p12}
    \|x_t - P_{j-1,k_0}(x_t)\|_{\real^2}\le 28\left(
      \sum_{k_0<k\le j} \delta_k^2\right)
    \dist(\pi_{j,k_0}\circ
    F_j^{-1}(x_t), \partial Q_{k_0,t}).
  \end{equation}
  Recalling~(\ref{eq:claim_j_p9})
  \begin{equation}
    \label{eq:claim_j_p13}
    \|x_t-q_t\|\le\|P_{j-1,k_0}(x_t)-q_t\|+28\times 16
    (1+\sum_{k\le k_0}\delta_k^2)^{1/2}\|P_{j-1,k_0}(x_t)-q_t\|,
  \end{equation}
  and the choice~(\ref{eq:claim_j_s1}) of the sequence
  $\{\delta_k\}_k$, we get:
  \begin{equation}
    \label{eq:claim_j_p14}
    \|x_t-q_t\|\le\frac{9}{8}\|P_{j-1,k_0}(x_t)-q_t\|.
  \end{equation}
  Now:
  \begin{equation}
    \begin{split}
      \label{eq:claim_j_p15}
                \|x_1-x_2\|&\ge\frac{1}{\glip P_{j-1,k_0}.}\|
                P_{j-1,k_0}(x_1)-P_{j-1,k_0}(x_2)
                \|\\
                &\ge\frac{\eta}{\glip P_{j-1,k_0}.}\max_{t=1,2}
                \dist(\pi_{j,k_0}\circ
                F_j^{-1}(x_t), \partial Q_{k_0,t})\\
                &\ge\frac{8\eta}{9\glip
                  P_{j-1,k_0}.\glip F_{k_0}.}\dist(x_t,F_{k_0}(\partial Q_{k_0,t})).
    \end{split}
  \end{equation}
  If $x_t+y_t\in\RadN j.$ then
    \begin{equation}
    \label{eq:claim_j_p16}
    \|(x_1+y_1) - (x_2+y_2)\|\ge\|x_1-x_2\|-c(\sigma_2)(\|y_1\|+\|y_2\|),
  \end{equation}
  where $\lim_{\sigma_2\to\infty}c(\sigma_2)=0$, and we can conclude
  as in \texttt{Step 1}.
  \par In the second case assume that $Q_{j-1,1}=Q_{j-1,2}$ but
  $Q_{j-1,1}\pardec\ne Q_{j-1,2}\pardec$. Then $P_{j-1}(x_1)$ and
  $P_{j-1}(x_2)$ lie on the same affine plane of $\dirset j-1.$
  and thus:
  \begin{equation}
    \label{eq:claim_j_p17}
    \|P_{j-1}(x_1)-P_{j-1}(x_2)\|\ge\max_{t=1,2}\dist(P_{j-1}(x_t),F_{j-1}(\partial Q_{k_0,t})),
  \end{equation}
  and we can then argue as in the first case.
  \par Third, if $Q_{j-1,1}\pardec= Q_{j-1,2}\pardec$ we can argue as
  in \texttt{Step 1}. In fact, by \texttt{Step 6} in Construction~\ref{constr:r4}
  the set $\dirnset j-1.$ is finite and, up to translations and
  dilations, there are only finitely many possibilities for the
  subcomplexes of $Y_j$ which project via $P_{j-1}$ onto some $F_{j-1}(Q)$
  for $Q\in\nwsqfm \addec,X_{j-1}.$. Thus we can find an $\alpha>0$
  such that if $\{Q_1,Q_2\}\subset\nwsqfm,
  X_j.$ are distinct, $x_t\in F(Q_t)$ ($t=1,2$) and $\pi_{j,j-1}(Q_1)$
  and $\pi_{j,j-1}(Q_2)$ belong to the same square of $\nwsqfm\addec,X_{j-1}.$ then
  \begin{equation}
    \label{eq:claim_j_p18}
    \|x_1-x_2\|\ge\alpha\max_{t=1,2}\dist(x_t,F_j(\partial Q_t)),
  \end{equation}
  and then argue as in \texttt{Step 1}.
\end{proof}
\begin{lem}[Compositions of $P_i$ are uniformly Lipschitz]
  \label{lem:unif_comp}
  Assume that $\sum_t\varepsilon_t<\infty$; then the Lipschitz maps
  $P_i:\RadN i.\to Y_i$ can be composed to give uniformly Lipschitz
  maps; specifically, for $k<i$ let:
  \begin{equation}
    \label{eq:unif_comp_s1}
    P_{i,k}=P_k\circ\cdots\circ P_{i-1}\circ P_i;
  \end{equation}
  then:
  \begin{equation}
    \label{eq:unif_comp_s2}
    \glip P_{i,k}.\le\prod_t(1+\varepsilon_t).
  \end{equation}
  Let $F_i(X_i)=Y_i$ ($i=\infty$ is admissible); then the maps
  \begin{equation}
    \label{eq:unif_comp_s3}
    P_{i,k}:\RadN i.\to Y_k
  \end{equation}
  as $i\nearrow\infty$ converge weak* to a map:
  \begin{equation}
    \label{eq:unif_comp_s4}
    P_{\infty,k}:Y_\infty\to Y_k
  \end{equation}
  which satisfies $\glip
  P_{\infty,k}.\le\prod_t(1+\varepsilon_t)$. For $k<l\le i$ ($i$ or
  $l$ can be $\infty$ with $P_{\infty,\infty}$ being taken to be the
  identity of $Y_\infty$) one has:
  \begin{equation}
    \label{eq:unif_comp_s5}
    P_{l,k}\circ P_{i,l} = P_{i,k};
  \end{equation}
  as $k\nearrow\infty$ $P_{\infty,k}$ converges weak* to
  $P_{\infty,\infty}$. 
\end{lem}
\begin{proof}
  Assuming that $\sum_t\varepsilon_t<\infty$ we have a uniform bound
  on the Lipschitz constants of the maps $P_{i,k}$:
  \begin{equation}
    \label{eq:unif_comp_p1}
    \sup_{i,k}\glip P_{i,k}.\le\prod_t(1+\varepsilon_t)<\infty.
  \end{equation}
  From the definition of $\RadN i.$ we get that if $\sigma_i>1$ (note
  that the $\sigma_i$'s are chosen very large in
  Lemma~\ref{lem:claim_j}) we have:
  \begin{equation}
    \label{eq:unif_comp_p2}
    \sup_{x\in\RadN i.}\|P_i(x)-x\|\le 100\cdot 5^{-i}.
  \end{equation}
  In particular, for a universal constant $C>0$ we have:
  \begin{equation}
    \label{eq:unif_comp_p3}
    \sup_l\sup_{x\in\RadN i+l.}\|P_{i+l,k}(x)-P_{i,k}(x)\|\le C5^{-i}.
  \end{equation}
  Therefore, on $Y_\infty$ the maps $P_{i,k}$ converge, uniformly as
  $i\nearrow\infty$ to a map $P_{\infty,k}$ which must be Lipschitz
  because of~(\ref{eq:unif_comp_p1}); the uniform
  bound~(\ref{eq:unif_comp_p1}) also ensures that convergence is in
  the weak* sense.
  \par From the definition of $P_{i,k}$ we have
  that~(\ref{eq:unif_comp_s5}) holds when all of  $\{i,l,k\}$ are
  finite. For $l=\infty$ or $i=\infty$ we establish the result by a
  limiting argument setting $P_{\infty,\infty}$ equal to the identity
  of $Y_\infty$. We are thus only left to show that $P_{\infty,k}$
  converges on $Y_k$ uniformly to the identity. But this is immediate
  observing that~(\ref{eq:unif_comp_p2}) gives:
  \begin{equation}
    \label{eq:unif_comp_p4}
    \sup_{x\in Y_\infty=\bigcap_k\RadN k.}\|P_{k,\infty}(x)-x\|\le
    10^3\times 5^{-k}.
  \end{equation}
\end{proof}
\begin{lem}[Convergence of the Embeddings]
  The topological embeddings $F_i:X_i\hookrightarrow\real^4$
  converge, as $i\nearrow\infty$, to a topological embedding
  $F_\infty:X_\infty\hookrightarrow\real^4$ such that:
  \begin{equation}
    \label{eq:embedd_conv_s1}
    16^{-1}\bigl(1+\sum_i\delta_i^2\bigr)^{1/2}\le\glip F_\infty.\le 23\bigl(1+\sum_i\delta_i^2\bigr)^{1/2}.
  \end{equation}
  For each $k<i$ ($i=\infty$ being admissible) one has a
  commutative diagram:
  \begin{equation}
    \label{eq:embedd_conv_s2}
    \xy
    (0,20)*+{X_{i}} = "xi"; (20,20)*+{Y_{i}} = "yi";
    (0,0)*+{X_k} = "xk"; (20,0)*+{Y_k} = "yk";
    {\ar "xi"; "yi"}?*!/_2mm/{F_{i}};
    {\ar "xk"; "yk"}?*!/^2mm/{F_k};
    {\ar "xi"; "xk"}?*!/^7mm/{\pi_{i,k}};
    {\ar "yi"; "yk"}?*!/_4mm/{P_{i,k}};
    \endxy
     \end{equation}
\end{lem}
\begin{proof}
  Note that from~(\ref{eq:constr_r4_r31}) we have:
  \begin{equation}
    \label{eq:embedd_conv_p1}
    \sup_{x\in X_{i+k}}\|F_{i+k}(x)-F_i(\pi_{i+k,i}(x))\|\le 200\cdot 5^{-i},
  \end{equation}
  and so the embeddings $F_i:X_i\hookrightarrow\real^4$ converge
  uniformly to a map $F_\infty:X_\infty\to\real^4$ which must
  satisfy~(\ref{eq:embedd_conv_s1}) because
  of~(\ref{eq:constr_r4_r30}). 
  \par The diagram~(\ref{eq:embedd_conv_s2}) commutes because
  of~(\ref{eq:constr_r4_32}) (the case $i=\infty$ is handled by a
  limiting argument).
  \par Finally, as $X_\infty$ is compact, in order to conclude that
  $F_\infty$ is an embedding it suffices to establish that it is
  injective. Let $x,y$ be distinct points of $X_\infty$; then for some
  $k$: $\pi_{\infty,k}(x)\ne \pi_{\infty,k}(y)$ and, as $F_k$ is an
  embedding:
  \begin{equation}
    \label{eq:embedd_conv_p2}
    F_k(\pi_{\infty,k}(x))\ne F_k(\pi_{\infty,k}(y));
  \end{equation}
  but as the diagrams~(\ref{eq:embedd_conv_s2}) commute:
  \begin{equation}
    \label{eq:embedd_conv_p3}
    P_{\infty,k}(F_\infty(x))\ne P_{\infty,k}(F_\infty(y)).
  \end{equation}
\end{proof}
\begin{lem}[Existence and nontriviality of the $2$-current]
  \label{lem:nontriv_curr_r4}
  The pushforward $F_{\infty\#}N_\infty$ is a nontrivial normal
  current in $\real^4$ supported on $Y_\infty$; in fact:
  \begin{equation}
    \label{eq:nontriv_curr_r4_s1}
    P_{\infty,0\#}F_{\infty\#}N_\infty=F_{0\#}N_0.
  \end{equation}
\end{lem}
\begin{proof}
  One just needs to prove~(\ref{eq:nontriv_curr_r4_s1}) and might argue from
  the commutative diagram~(\ref{eq:embedd_conv_s2}) for $(i,k)=(\infty,0)$. But some sleight of
  hand is concealed in this approach and for the Apprehensive Analyst we
  provide a direct computation which uses weak* continuity of normal currents:
  \begin{equation}
    \label{eq:nontriv_curr_r4_p1}
    P_{\infty,0\#}F_{\infty\#}N_\infty(fdg_1\wedge df_2)=N_\infty(f\circ
    P_{\infty,0}\circ F_\infty d(g_1\circ
    P_{\infty,0}\circ F_\infty)\wedge  d(g_2\circ
    P_{\infty,0}\circ F_\infty));
  \end{equation}
  but $P_{i,0}\circ F_i\circ\pi_{\infty,i}\xrightarrow{\text{w*}}
  P_{\infty,0}\circ F_\infty$ as
  $i\nearrow\infty$ and thus:
  \begin{equation}
    \label{eq:nontriv_curr_r4_p2}
    \begin{split}
      P_{\infty,0\#}\circ F_{\infty\#}N_\infty
      &=\lim_{i\to\infty}N_\infty\bigl[
      (P_{i,0}\circ F_i\circ\pi_{i,\infty})^*fdg_1\wedge dg_2
      \bigr]\\
      &=\lim_{i\to\infty}N_i((P_{i,0}\circ F_i)^*fdg_1\wedge dg_2)\\
      &=\lim_{i\to\infty}P_{i,0\#}F_{i\#}N_i(fdg_1\wedge
      dg_2)=F_{0\#}N_0(fdg_1\wedge dg_2).
    \end{split}
  \end{equation}
\end{proof}
\begin{thm}[$2$-unrectifiability of $Y_\infty$]
  \label{thm:unrect_r4}
  $Y_\infty$ is purely $2$-unrectifiable in the sense that whenever
  $K\subset\real^2$ is compact and $\Phi:K\to \real^4$ is Lipschitz,
  $\hmeas 2.(\Phi^{-1}(Y_\infty)\cap K)=0$.
\end{thm}
\begin{proof}
  We will argue by contradiction assuming that $K\subset
  \Phi^{-1}(Y_\infty)$ and that $\hmeas
  2.(K)>0$. The main difference from the proof of
  Theorem~\ref{thm:unrect_hilbert} is \texttt{Step 1} where we resort
  to a weak* (approximate) lower-semicontinuity argument.
  \par\noindent\texttt{Step 1: Reduction to the case in which $\Phi$
    is a graph over $Y_0$.}
  \par Let $\Phi_n=P_{\infty,n}\circ \Phi$, which are well-defined and
  uniformly Lipschitz. By Lemma~\ref{lem:unif_comp} we also have that
  $\Phi_n$ converges weak* to $\Phi$.
  \par We now consider the Borel set $E\subset K$
  consisting of those points which are Lebesgue density points of the
  set of points where $\Phi$ and each $\Phi_n$ is differentiable and where $d\Phi_0$ has
  rank $<2$; our goal is to show that
  \begin{equation}
    \label{eq:unrect_r4_p1}
    \hmeas 2.(\im\Phi \cap P^{-1}_{\infty,0}(\Phi_0(E)))=0.
  \end{equation}
First, the area formula \cite[Thm.~5.1]{ambrosio-rectifiability}
gives $\hmeas 2.(Y_0\cap\Phi_0(E)) =
0$. Secondly, for each $n$, using the square complex structure of
$\{X_i\}_{i\le n}$, the set $Y_n\setminus F_n(\oskel X_n.)$ can be partitioned into countably
many closed sets $\{S_\alpha\}_\alpha$ (e.g.~taking each $F_n(Q)$ for
$Q\in\nwsqfm\addec,X_n.$)
such that each restriction
$P_{n,0}|_{S_\alpha}:S_{\alpha}\to P_{n,0}(S_\alpha)$ is biLipschitz, thus
giving:
\begin{equation}
  \label{eq:unrect_r4_p2}
  \hmeas 2.(Y_n\cap P^{-1}_{n,0}(\Phi_0(E)))=0.
\end{equation}
In particular, the area formula implies that:
\begin{equation}
  \label{eq:unrect_r4_p3}
  \int_K\chi_EJ_2(d\Phi_n)\,d\hmeas 2.=0. 
\end{equation}
We want to uset the lowersemicontinuity of the area
functional (see for example~\cite[Subsec.~2.6]{ambrosio-bvbook}), but we need the
domain of the maps $\Phi_n$, $\Phi_\infty$ to be open. Fix
$\varepsilon>0$ and choose $U\subset E$ open with $\hmeas
2.(U\setminus E)<\varepsilon$. By McShane's Lemma we can extend each
$\Phi_n$ to a $7C$-Lipschitz map $\tilde\Phi_n:U\to\real^4$ which
coincides on $E$ with $\Phi_n$. Up to passing to a subsequence we can
assume $\tilde\Phi_n\xrightarrow{\text{w*}}\tilde\Phi_\infty$ were
$\tilde\Phi_\infty|E=\Phi_\infty$. We can now invoke lower-semicontinuity
of area:
\begin{equation}
  \label{eq:unrect_r4_p4}
  \begin{split}
    \int_K\chi_EJ_2(d\Phi_\infty)\,d\hmeas
    2.\le&\int_UJ_2(d\tilde\Phi_\infty)\,d\hmeas 2.
    \le\liminf_{n\to\infty}\int_UJ_2(d\tilde \Phi_n)\,d\hmeas 2.\\
    &\le\limsup_{n\to\infty}\int_{U\setminus E}J_2(d\tilde
    \Phi_n)\,d\hmeas 2.+
    \limsup_{n\to\infty}\int_{E}J_2(d\Phi_n)\,d\hmeas 2.\\
    &\le 49C^2\varepsilon,
  \end{split}
\end{equation}
and~(\ref{eq:unrect_r4_p1}) follows letting $\varepsilon\searrow0$ and
applying the area formula.
\par\noindent\texttt{Step 2: Existence of square holes.}
\par The same argument as in \texttt{Step 2} of
Theorem~\ref{thm:unrect_hilbert} goes through with minor modifications. 
\par First, the (generalized) square-complex structure of
$X_n\setminus\oskel X_n.$ induces a generalized square-complex
structure on $Y_n\setminus F_n(\oskel X_n.)$ via the homeomorphism
$F_n$: thus, in the following, we will implicitly identify $\nwsqfm
k,X_n.$ (resp.~$\nwsqfm\addec,X_n.$) with $\nwsqfm
k,Y_n.$ (resp.~$\nwsqfm\addec,Y_n.$).
\par Second, compared to the $l^2$-case there are differences in
indexing the $\nwsqfm *,Y_n.$, $\nwsqfm *,X_n.$. In fact, as the
construction is no longer self-similar, $\nwsqfm k,X_n.$ does not
represent the set of squares of $X_n$ of generation $k$ (and side
length $5^{-k}$), but the set of squares obtained by subdividing each
square of $\nwsqfm ,X_n.$ $k$-times (and so the side length is
$5^{-k}$-times the side length of the parent square in $\nwsqfm
,X_n.$). Moreover, we need a notation for the set of squares obtained
by subdividing each square of $\nwsqfm\addec, X_n.$ $k$-times: we will
use $\nwsqfm{\addec,k},X_n.$.
\par Third, in \texttt{Step 1} of Construction~\ref{constr:r4} we took
a piecewise-affine approximation of $\Psi_\delta$ which involved
subdividing squares $N$-extra times. We must thus modify the
definition of $i_n$~(\ref{eq:unrect_hilbert_p6}) letting:
\begin{equation}
    \label{eq:unrect_r4_p5}
    i_n = \lceil -\log_{5}(5^{-n-N-5}c\delta_n)\rceil.
  \end{equation}
\par Fourth, we have to consider a square $Q\in
P_{n-1,0}(\nwsqfm{\addec,i_n-n},Y_{n-1}.)$ and partition $\hat Q\adec$
into $\simeq 5^{i_n-n}$ annuli consisting of squares of
$P_{n-1,0}(\nwsqfm{\addec,i_n-n},Y_{n-1}.)$. Having fixed such an
annulus $A$, the goal is again to show that $K=\dom\Phi\subset Y_0$
(we have reduced to the case in which $\Phi$ is a graph over a subset
of $Y_0$ in the previous \texttt{Step 1}) has to miss one of
the squares of $A$.
\par Then the proof follows the same logic as in \texttt{Step 2} of
Theorem~\ref{thm:unrect_hilbert} with some minor notational
modifications:
\begin{itemize}
\item $\sqfm j,Y_j.$ becomes $\nwsqfm ,Y_j.$, compare the previous
  discussion about idexing.
\item $\sqfm i_n,Q\zerodec_{n-1}.$ becomes $\nwsqfm
  i_n-n,Q\zerodec_{n-1}.$, where $\nwsqfm k,Q.$ denotes the set of
  subsquares of $Q$ obtained by taking $k$-iterated subdivisions.
\item We cannot simply use the projejction $P_0$, but must use
  $P_{j,0}$ when projecting points from $Y_j$ to $Y_0$. In particular,
  instead of writing $Q\jdec_{i_n,\beta}=P_0^{-1}
(Q\zerodec_{i_n,\beta})\cap
Q\jdec_{j,\alpha}$, we need to consider $Q\jdec_{i_n,\beta}=P_{j,0}^{-1}
(Q\zerodec_{i_n,\beta})\cap
Q\jdec_{j,\alpha}$.
\end{itemize}
\par\noindent\texttt{Step 3: The choice of the $\delta_k$'s.}
\par Here we have to guarantee that~(\ref{eq:claim_j_s1}) holds; this
can be achieved by shifting the sequence we used in
Theorem~\ref{thm:unrect_hilbert} to the right:
\begin{equation}
  \label{eq:eq:unrect_r4_p6}
  \delta_k=\frac{1}{10^9+k}.
\end{equation}
\end{proof}
\section{$k$-current in $\real^{k+2}$}
\label{sec:krk}
The $k$-current in $\real^{k+2}$ is constructed resorting to a trick that was already employed
in \cite[Sec.~4]{schioppa_metric_sponge}: once one is able to
construct a $2$-current which meets 
all Lipschitz surfaces
which are graphs over a coordinate plane in a $\hmeas 2.$-null set, one can iterate
over all planes parallel to a pair of coordinate axes. In the
following we let $\{e_\xi\}_{1\le \xi\le l}$ denote the standard
orthonormal basis of $\real^l$ (where $l=k$ or $l=k+2$)
and for $\xi<\zeta$ we let $e_\xi\oplus e_\zeta$ denote the plane
spanned by $e_\xi$ and $e_\zeta$. Finally, we will identify the set of
planes $\{e_\xi\oplus e_\zeta\}_{1\le \xi<\zeta\le k}$ with
$\zahlen_{{k\choose 2}}$ and we will write equations like $s=e_\xi\oplus
e_\zeta\mod {{k\choose 2}}$ or $e_\xi\oplus e_\zeta=2\mod
\binom{k}{2}$.
\def\mpj{\text{\normalfont pj}_{\xi,\zeta}}
\begin{constr}[Modifications to Construction~\ref{constr:basic_map}]
   \label{constr:rk_1}
Now Construction~\ref{constr:basic_map} is generalized adding an additional parameter:
a 2-plane $e_\xi\oplus e_\zeta$. Let $k$ be a $k$-cube isometric to
$[0,5^{-i}]$ and let $\mpj$ denote the projection onto
$e_\xi\oplus e_\zeta$ and set $Q=\mpj(K)$. Let $Q\adec$, $Q\cdec$,
$Q\odec$, $\tilde Q$, etc\dots as in
Construction~\ref{constr:basic_map} and set:
\begin{equation}
  \label{eq:constr_rk1_1}
  \begin{aligned}
    K\adec&=\mpj^{-1}(Q\adec)\\
    K\odec&=\mpj^{-1}(Q\odec)\\
    K\cdec&=\mpj^{-1}(Q\cdec).
  \end{aligned}
\end{equation}
We use standard covering theory to find a double cover
$\tilde\pi:\tilde K\adec\to K\adec$, and a lift $\widetilde{\mpj}:
\tilde K\adec\to \tilde Q\adec$
such that the following diagram commutes:
\begin{equation}
  \label{eq:constr_rk1_2}
    \xy
    (0,20)*+{\tilde K\adec} = "tka"; (20,20)*+{K\adec} = "ka";
    (0,0)*+{\tilde Q\adec} = "tqa"; (20,0)*+{Q\adec} = "qa";
    {\ar "tka"; "ka"}?*!/_2mm/{\tilde\pi};
    {\ar "tqa"; "qa"}?*!/_2mm/{\tilde\pi_Q};
    {\ar "tka"; "tqa"}?*!/^4mm/{\widetilde{\mpj}};
    {\ar "ka"; "qa"}?*!/_4mm/{\mpj};
    \endxy
\end{equation}
where $\tilde\pi_Q:\tilde Q\adec\to Q\adec$ is the double cover from
Construction~\ref{constr:basic_map}.
We then glue $\tilde K\adec$ 
back to $K\odec\cup K\cdec$ by gluing together the pair of points of
$\partial\tilde K\adec$ that are mapped to the same point by
$\tilde\pi$. If $\tilde K$ denotes the resulting cube-complex,
then $\tilde\pi$ extends to a branched covering $\tilde\pi:\tilde K\to
K$ and we also obtain an extension $\widetilde{\mpj}:\tilde K\to
\tilde Q$ of $\widetilde{\mpj}|\Inside(\tilde K\adec)$ which makes the 
following diagram commute:
\begin{equation}
  \label{eq:constr_rk1_3}
    \xy
    (0,20)*+{\tilde K} = "tka"; (20,20)*+{K} = "ka";
    (0,0)*+{\tilde Q} = "tqa"; (20,0)*+{Q} = "qa";
    {\ar "tka"; "ka"}?*!/_2mm/{\tilde\pi};
    {\ar "tqa"; "qa"}?*!/_2mm/{\tilde\pi_Q};
    {\ar "tka"; "tqa"}?*!/^4mm/{\widetilde{\mpj}};
    {\ar "ka"; "qa"}?*!/_4mm/{\mpj};
    \endxy
\end{equation}
we then obtain $\Psi:\tilde K\to\real^2$ as the composition
$\Psi=\Psi_{\tilde Q}\circ\widetilde{\mpj}$ where $\Psi_{\tilde
  Q}:\tilde Q\to\real^2$  
is the map we built in Construction~\ref{constr:basic_map}.
\end{constr}
\def\nwcell#1,#2.{\setbox1=\hbox{$#1$\unskip}\setbox2=\hbox{$#2$\unskip}\text{\normalfont
    Cell}\ifdim\wd1>0pt_{#1}\fi({\ifdim\wd2>0pt #2\else X_j\fi}) }
\def\kskel#1.{\setbox1=\hbox{$#1$\unskip}\text{\normalfont
    Sk}_{k-1}({\ifdim\wd1>0pt #1\else X_{j-1}\fi})}
\begin{constr}[Modification to Construction~\ref{constr:r4}]
  \label{constr:rk_2}
  \par\noindent\texttt{Step 1: Piecewise affine approximation.}
  \par For fixed $\delta,\xi,\zeta$, let $\Psi_\delta:\tilde
  K\to\real^2$ be as in Construction~\ref{constr:rk_2} using the
  parameters $\delta$, $e_{\xi}\oplus e_\zeta$.
If $\tilde K^{(m)}$ denotes the $m$-th iterated subdivision of $\tilde
K$, we can find $N\in\natural$ and a piecewise affine 
approximation $\Phi_\delta:\tilde K^{(N)}\to\real^2$ of $\Psi_\delta$
such that the following analogs of~(\ref{eq:constr_r4_3}),
(\ref{eq:constr_r4_4}) and (\ref{eq:constr_r4_5}) hold:
\begin{align}
  \label{eq:constr_rk2_1}
  \glip \Phi_\delta.&\in\left[\frac{\delta}{16},23\delta\right]\\
  \label{eq:constr_rk2_2}
  \|\Phi_\delta(p_1)-\Phi_\delta(p_2)\| &\ge\frac{\delta}{3}\phi(r(\mpj(p_1)))\\
  \label{eq:constr_rk2_3}
  \|\Phi_\delta\|&\le 2\delta\diam K.
\end{align}
We let $X_0=[0,1]^k$ and $F_0:X_0\to\bigoplus_{1\le\xi\le
  k}e_\xi\subset\real^{k+2}$ denote the standard
isometric embedding. We obtain $X_1$ from $X_0$ by applying
Construction~\ref{constr:rk_1}
with $e_\xi\oplus e_\zeta=0\mod\binom{k}{2}$ and then let 
\begin{equation}
  \label{eq:constr_rk2_4}
  F_1 = F_0\circ \pi_{1,0} + \Phi_{\delta_1}\otimes(e_{k+1}\oplus e_{k+2}).
\end{equation}
\par\noindent\texttt{Step 2: Construction of $X_{j+1}$ and $F_{j+1}$.}
\par We need first to generalize the notation. We let $\nwcell ,.$ denote the set of k-dimensional cells
of $X_j$; while $X_1$ is a $k$-cube complex,
as in Construction~\ref{constr:r4}, $X_j$ does not have a $k$-cube complex structure, 
but it is a union of its $k$-cells $\nwcell ,.$ away from the
$(k-1)$-skeleton $\kskel .$ of $X_{j-1}$, where $\kskel .$ embedds
isometrically in $X_j$. Moreover, we let $\kskel X_j.=\kskel
.\cup\bigcup_{K\in\nwcell ,.}\partial K$; in particular:
\begin{equation}
   \label{eq:constr_rk2_5}
   X_j\setminus\kskel X_j. = \bigcup_{K\in\nwcell ,.}\Inside(K).
\end{equation}
For $K\in\nwcell ,.$ we define the radial basis function
\begin{equation}
  \label{eq:constr_rk2_6}
  \varphi_K(x) = 
  \begin{cases}
    \exp\left(
      -\frac{\sigma_j}{\dist(\pi_{\tau(K)}(x), F_j(\partial K))}
    \right)\times 46\diam(F_j(K))&\text{if
      $\pi_{\tau(K)}\in F_j(\Inside(K))$}\\
    0&\text{otherwise,}
  \end{cases}
\end{equation}
where $\pi_{\tau(K)}$ denotes the orthogonal projection onto the
affine $k$-plane $\tau(K)$ containing $F_j(K)$.
We then define
the radial basis neighbourhood $\RadN j.$ as:
\begin{multline}
  \label{eq:constr_rk2_7}
  \RadN j. = \biggl\{
  p\in\real^{k+2}:\text{there is a $K\in\nwcell ,.$}: p = x+y, x\in
  F_j(K),\\
  y\perp
  \tau(K), \text{and $\|y\|\le\varphi_K(x)$}
\biggr\}.
\end{multline}
We then define $P_j:\RadN j.\to Y_j$ by $p=x+y\mapsto X$ and, as in
Section~\ref{sec:2r4}, it follows that:
\begin{description}
\item[(Claim $j$)] For each $\varepsilon_j>0$ there is a $\sigma_j>0$
  such that $P_j$ is $(1+\varepsilon_j)$-Lipschitz. 
\end{description}
Let $\nwcell m,X_j.$ denote the set of cells obtained by subdividing
each cell of $\nwcell ,.$ $m$-times, and let:
\begin{equation}
  \label{eq:constr_rk2_8}
  \nwcell\infty,. = \bigcup_{m\ge1}\nwcell m,..
\end{equation}
Now a cell $K\in\nwcell\infty,.$ is adapted to $\RadN j.$ if the
$(23\delta_j\diam F_j(K))$-neighborhood of $F_j(K)$ is contained in
$\RadN j.$ and if, denoting by $\PAR(K)\in\nwcell , X_j.$ the unique
cell containing $K$, one has:
\begin{equation}
  \label{eq:tritanomaly-aug-add3}
  \max_{x\in K}\dist(x,\partial K) \le \delta_j\max_{x\in K}\dist(x,\partial\PAR(K)).
\end{equation} We let $\nwcell \addec,.$ denote the set of maximal
adapted $k$-cubes of $\nwcell\infty,.$; the elements of $\nwcell
\addec,.$ have pairwise disjoint interia and satisfy:
\begin{equation}
  \label{eq:constr_rk2_9}
  X_j \setminus\kskel X_j. = \bigcup_{K\in\nwcell\addec,X_j.}\Inside(K).
\end{equation}
Fix $e_\xi\oplus e_\eta = j\mod\binom{k}{2}$ and apply
Construction~\ref{constr:rk_1} to each $K\in\nwcell\addec,.$ to get
$\Phi_{K,\delta_{j+1}}:\tilde K\to\real^2$. As in
Construction~\ref{constr:r4} we can ensure that if $K_1\ne K_2$
$\Phi_{K_1,\delta_{j+1}}$ and $\Phi_{K_2,\delta_{j+1}}$ can be taken
to differ up to composition with translations and dilations.
Let $\dirnset j.=\bigcup_{K\in\nwcell ,.}\tau_0(K)$ where $\tau_0(K)$
denotes the $k$-plane parallel to $\tau(K)$ and passing through the
origin. By induction we assume $\dirnset j.$ to be finite and choose a
finite set of pairs $\{(e_{1,K},e_{2,K})\}_{K\in\nwcell\addec,.}$ such
  that each $(e_{1,K},e_{2,K})$ is an orthonormal basis of the
  orthogonal complement of $\tau_0(\PAR(K))$ where $\PAR(K)\in\nwcell
  ,.$ is the $k$-cell containing $K$. We can then define:
\begin{equation}
  \label{eq:constr_rk2_10}
  F_{j+1}(x) = F_j\circ\pi_{j+1}(x) + \sum_{K\in\nwcell
    \addec,.}\Phi_{K,\delta_{j+1}}(x)\otimes(e_{1,K}\oplus e_{2,K}),
\end{equation}
and get 
\begin{equation}
  \label{eq:constr_rk2_11}
  \glip F_{j+1}.\in\left[16^{-1}(1+\sum_{l\le j+1}\delta_l^2)^{1/2}, 23(1+\sum_{l\le j+1}\delta_l^2)^{1/2}\right].
\end{equation}
\end{constr}
As in the $\real^4$-case we let $Y_i=F_i(X_i)$ and
$Y_\infty=F_\infty(X_\infty)$.
\begin{thm}[$2$-unrectifiability of $Y_\infty\subset\real^{k+2}$]
  \label{thm:unrect_rk}
  $Y_\infty$ is purely $2$-unrectifiable in the sense that whenever
  $K\subset\real^2$ is compact and $\Phi:K\to \real^{k+2}$ is Lipschitz,
  $\hmeas 2.(\Phi^{-1}(Y_\infty)\cap K)=0$.
\end{thm}
\def\mpjj{\text{\normalfont pj}_{\xi_j,\zeta_j}}
\begin{proof}
  We will focus on the differences with the proof of
  Theorem~\ref{thm:unrect_r4}.
\par\noindent\texttt{Step 1: Reduction to the case in which $\Phi$
    is a graph over $Y_0$.}
  \par Let $\Phi:K\subset[0,1]^2\to Y_\infty$ be Lipschitz with $\hmeas
  2.(K)>0$. Let $\Phi_n=P_{\infty,n}\circ\Phi$ and $E\subset K$ be the
  set of differentiability points $p$ of $\{\Phi_n\}_n,\Phi$ such that for each pair $(\xi,\zeta)$ with
  $1\le\xi<\zeta\le k$ (note the $k$, not $k+2$! Our construction has
  already screwed-up the behavior in the last two coordinates):
  \begin{equation}
    \label{eq:unrect_rk_p1}
    \det
    \begin{pmatrix}
      \langle e_\xi, d\Phi_0(\partial_x)\rangle & 
      \langle e_\xi, d\Phi_0(\partial_y)\rangle \\
      \langle e_\zeta, d\Phi_0(\partial_x)\rangle & 
      \langle e_\zeta, d\Phi_0(\partial_y)\rangle
    \end{pmatrix}(p)=0.
  \end{equation}
  As $Y_0$ lies in $\bigoplus_{\xi\le k}e_\xi$ the area formula gives:
  \begin{equation}
    \label{eq:unrect_rk_p2}
    \hmeas 2.(Y_0\cap \Phi_0(E))=0.
  \end{equation}
  Now, using that $\Phi_n\xrightarrow{\text{w*}}\Phi$ and the weak*
  lower-semicontinuity of the area functional as in \texttt{Step 1} of
  Theorem~\ref{thm:unrect_r4} we conclude that:
  \begin{equation}
    \label{eq:unrect_rk_p3}
    \hmeas 2.(Y_\infty\cap \Phi(E))=0.
  \end{equation}
  Thus, up to passing to a countable partition of $K$ and throwing
  away an $\hmeas 2.$-null set we can assume that there are
  $1\le\xi_0<\zeta_0\le k$ such that for each $p\in K$:
  \begin{equation}
    \label{eq:unrect_rk_p4}
    \det
    \begin{pmatrix}
      \langle e_{\xi_0}, d\Phi_0(\partial_x)\rangle & 
      \langle e_{\xi_0}, d\Phi_0(\partial_y)\rangle \\
      \langle e_{\zeta_0}, d\Phi_0(\partial_x)\rangle & 
      \langle e_{\zeta_0}, d\Phi_0(\partial_y)\rangle
    \end{pmatrix}(p)\ne 0.
  \end{equation}
  Using \cite[Thm.~9]{kirchheim_metric_diff} in, which is essentially a 
measurable and Lipschitz version of the Inverse Function Theorem, up
to further partioning and throwing away an $\hmeas 2.$-null set we are
reduced to the case $K\subset\pi_{e_{\xi_0}\oplus e_{\zeta_0}}(Y_0)$
where $\pi_{e_{\xi_0}\oplus e_{\zeta_0}}$ denotes the orthogonal
projection onto $e_{\xi_0}\oplus e_{\zeta_0}$.
\par\noindent\texttt{Step 2: Existence of square holes.}
\par The proof now proceeds as in \texttt{Step 2} of
Theorems~\ref{thm:unrect_r4}, \ref{thm:unrect_hilbert} but we spell
out more details because we deal both with squares and $k$-dimensional
cells. 
\par Let $n-1=e_{\xi_0}\oplus e_{\zeta_0}\mod\binom{k}{2}$ and let
$Q\in\pi_{e_{\xi_0}\oplus
  e_{\zeta_0}}(P_{n-1,0}(\nwcell\addec,Y_{n-1}.))$ where 
\begin{equation}
  \label{eq:unrect_rk_p5}
  i_n = \lceil -\log_{5}(5^{-n-N-5}c\delta_n)\rceil;
\end{equation}
let $\hat Q\adec$ be as in Construction~\ref{constr:basic_map} 
and partition $\hat Q\adec$ into $\approx 5^{i_n-n}$ annuli
consisting of squares of $\sqfm i_n-n,Q.$ (i.e.~subdivide $Q$ into 
25 subsquares $(i_n-n)$-times). We consider one such an
annulus $A$. Our goal is to show that $K$ has to miss the interior of one of
the squares in $A$. Let $p_{\alpha}$, $p_{\alpha+1}$ be as in
\texttt{Step 2} of the proof of Theorem~\ref{thm:unrect_r4}, and
we will show that $\Phi_j(p_\alpha)$ and $\Phi_j(p_{\alpha+1})$ belong to
the same cell of $\nwcell\addec,Y_j.$. This is true by construction
when $j=0$ and for $j\ge 1$ we assume by induction that $\Phi_{j-1}(p_\alpha)$,
$\Phi_{j-1}(p_{\alpha+1})$ belong to the same $K\jonedec_{j-1}\in\nwcell
\addec,Y_{j-1}.$. 
Let $K\jdec_{j,\beta}\in\nwcell \addec,Y_j.$ denote the
cell containing $\Phi_j(p_\beta)$ and assume by contradiction that
$K\jdec_{j,\alpha}\ne K\jdec_{j,\alpha+1}$. In the following we will
use the decorators ${}\adec$, ${}\odec$, ${}\cdec$ and $\hat{}$\ as in
Constructions~\ref{constr:basic_map} and \ref{constr:rk_1}: for example $\hat
K_{j-1,\apoldec}\jonedec$ is obtained as $\hat K\adec$ if we let
$K=K\jonedec_{j-1}$. In particular, as $K\jdec_{j,\alpha}\ne
K\jdec_{j,\alpha+1}$ we must have $P_{j-1}(K\jdec_{j,\beta})\subset
K\jonedec_{j-1,\apoldec}$. Let now
$Q\zerodec_{i_n,\beta}\in\pi_{e_{\xi_0}\oplus
  e_{\zeta_0}}(P_{j-1,0}(\nwcell{i_n-j+1},Y_{j-1}.))$ be the 
square containing $p_\beta$. Note that $Q\zerodec_{i_n,\beta}\subset Q$ can be
identified with a square of an iterated subdivision of $Q$, more
precisely, $Q\zerodec_{i_n,\beta}\in\nwsqfm i_n-n,Q.$. Let
$K\jdec_{i_n,\beta}=P_{j,0}^{-1}(Q\zerodec_{i_n,\beta})\cap
K\jdec_{j,\beta}$, and let $q\jdec_\beta$ denote the center of the
cell $K\jdec_{i_n,\beta}$. As $\Phi$ is $C$-Lipschitz:
\begin{equation}
  \label{eq:unrect_rk_p6}
  d(\Phi_j(p_\alpha), \Phi_j(p_{\alpha+1}))\le 4\sqrt{k}C\times c\delta_n\diam
  Q.
\end{equation}
As $F_j$ is $\glip F_\infty.$-Lipschitz and as $\diam
F_j^{-1}(K\jdec_{j,\beta})\le 2c\delta_n\sqrt{k}\diam Q$,
\begin{align}
  \label{eq:unrect_rk_p7}
  d(q\jdec_\beta,\Phi_j(p_\beta))&\le2\sqrt{k}\glip F_\infty.\times
  c\delta_n\diam Q\\
  \label{eq:unrect_rk_p8}
  d(q\jdec_\alpha, q\jdec_{\alpha+1})&\le4(\sqrt{k}C + \glip
  F_\infty.)\times c\delta_n\diam Q.
\end{align}
Let $S\jonedec_{j-1}=F^{-1}_{j-1}(K\jonedec_{j-1})$ and
$S\jdec_{j,\beta}=F^{-1}_{j}(K\jdec_{j,\beta})$; we must have
$S\jdec_{j,\alpha}\ne S\jdec_{j,\alpha+1}$ and
$\pi_{j-1}(F_j^{-1}(q_\beta\jdec))\in S\jonedec_{j-1,\apoldec}$. Note that
$F^{-1}_j(q\jdec_\beta)$ must be at distance $\ge 5^{-3}\diam(\partial
S\jonedec_{j-1,\apoldec})$ from $\partial
S\jonedec_{j-1,\apoldec}$ so that
\begin{equation}
  \label{eq:unrect_rk_p9}
  \phi(r(\mpjj(F_j^{-1}(q\jdec_\beta))))\ge 5^{-3}\diam(\partial
S\jonedec_{j-1,\apoldec}),
\end{equation}
where $e_{\xi_j}\oplus e_{\zeta_j}=j-1 \mod\binom{k}{2}$. As
$F_j^{-1}(q\jdec_\alpha)\ne F_j^{-1}(q\jdec_{\alpha+1})$, they belong
to different sheets of the double cover, and as
$\pi_{j-1}(S\jdec_{j,\alpha})$ and $\pi_{j-1}(S\jdec_{j,\alpha+1})$
are adjacent, we let $\hat q\jdec_\alpha$ be the center of the cell
of $\nwcell i_n-j,.$ adjacent to $K\jdec_{i_n,\alpha+1}$ and such that
$\pi_{j-1}(F_j^{-1}(q\jdec_\alpha)) = \pi_{j-1}(F_j^{-1}(\hat
q\jdec_\alpha))$. We now have:
\begin{align}
  \label{eq:unrect_rk_p10}
  r(\mpjj(F_j^{-1}(q\jdec_\alpha)) &= r(\mpjj(F_j^{-1}(\hat
  q\jdec_\alpha))\\
  \label{eq:unrect_rk_p11}
  \left|
     \theta(\mpjj(F_j^{-1}(q\jdec_\alpha)) -  \theta(\mpjj(F_j^{-1}(q\jdec_\alpha))
  \right| &= \pi.
\end{align}
Invoking~(\ref{eq:constr:basic_map_1}) we get:
\begin{equation}
  \label{eq:unrect_rk_p12}
  d(q\jdec_\alpha,\hat q\jdec_\alpha) \ge\frac{5^{-3}}{2}\delta_j\diam(\partial
S\jonedec_{j-1,\apoldec})\ge\frac{5^{-3}}{2\glip
  F_0.}\delta_n\diam Q, 
\end{equation}
where we used that $F_\infty$ and the maps $P_{\infty,n}$ are Lipschitz and that $Q$ lies in the
$F_0$-image of $S\jonedec_{j-1,\apoldec}$. But as $\hat q_\alpha\jdec$
is the center of the cell of $\nwcell i_n-j,.$ adjacent to
$K\jdec_{i_n,\alpha+1}$ we get:
\begin{equation}
  \label{eq:unrect_rk_p13} 
  d(q_\alpha\jdec, \hat q_\alpha\jdec) \le 16(\sqrt{k}C + \glip
  F_\infty.)\times c\delta_n\diam Q.
\end{equation}
Thus, if $c$ is chosen sufficiently small in function of $\sqrt{k}, C,
\glip F_\infty.$ we obtain a contradiction and conclude that
$K\jdec_{j,\alpha}=K\jdec_{j,\alpha+1}$. A consequence of this
discussion, specialized to $j=n$, is that $\Phi_n(p_1)$ and
$\Phi_n(p_t)$ belong to the same sheet of the double cover
$P_{n-1}^{-1}(\hat K^{(n-1)}_{n-1,\apoldec})\cap Y_n\to \hat
K^{(n-1)}_{n-1,\apoldec}$ while the choice of $c$ gives:
\begin{equation}
  \label{eq:unrect_rk_p14} 
  d(F_n^{-1}(\Phi_n(p_1)), F_n^{-1}(\Phi_n(p_t)))\le 5^{-3}\diam Q.
\end{equation}
Note, however, that as $n-1=e_{\xi_0}\oplus
e_{\zeta_0}\mod\binom{k}{2}$, from the definition of $\Psi$ in
Construction~\ref{constr:rk_1} and (\textbf{ShSep}) in
Construction~\ref{constr:basic_map} we get a contradiction. Thus $K$
misses one of the squares of the annulus $A$.
\end{proof}
\bibliographystyle{alpha}
\bibliography{sponge2}
\end{document}